\documentclass{article}

\PassOptionsToPackage{numbers,compress}{natbib}
 \usepackage[preprint]{neurips_2026}

\usepackage{amsmath}
\usepackage[utf8]{inputenc} 
\usepackage[T1]{fontenc}    
\usepackage{url}            
\usepackage{booktabs}       
\usepackage{amsfonts}       
\usepackage{nicefrac}       
\usepackage{microtype}      
\usepackage{xcolor}         
\usepackage{amsthm}
 \usepackage{stmaryrd} 
\usepackage{tkz-tab}
\usepackage{listings}
\usepackage{caption}  
\usepackage{hyperref}
\usepackage[capitalize]{cleveref}
\usepackage{subcaption} 
\usepackage{todonotes}
\usepackage{graphicx}

\newtheorem{theorem}{Theorem}[section]
\newtheorem{proposition}{Proposition}[section]
\newtheorem{definition}{Definition}[section]
\newtheorem{remark}{Remark}[section]

\newcommand{\N}{\mathbb{N}}

\newcommand{\R}{\mathbb{R}}
\renewcommand{\S}{\mathbb{S}}

\newcommand{\Om}{\Omega}
\newcommand{\dOm}{{\partial \Omega}}

\newcommand{\eps}{\varepsilon}
\newcommand{\Vol}{\text{Vol}}
\newcommand{\Per}{\text{Per}}
\newcommand{\Jac}{\text{Jac}}

\DeclareMathOperator{\tr}{tr}

\title{Parametrizing Convex Sets Using Sublinear Neural Networks}

\author{%
  Eloi Martinet \\
  Institute for Mathematics,\\
  University of Würzburg\\
  Germany \\
  \texttt{eloi.martinet@uni-wuerzburg.de} \\
}

\begin{document}

\maketitle

\begin{abstract}
  We propose a neural parameterization of convex sets by learning sublinear (positively homogeneous and convex) functions. Our networks implicitly represent both the support and gauge functions of a convex body. We prove a universal approximation theorem for convex sets under this parametrization. Empirically, we demonstrate the method on shape optimization and inverse design tasks, achieving accurate reconstruction of target shapes. 
\end{abstract}

\section{Introduction}

\subsection{Convex Shape Optimization}

Shape optimization under convexity constraints has been extensively studied due to its analytical tractability and practical relevance \cite{buttazzo1997shape}. Early numerical approaches enforced convexity through penalization, notably via the distance to the convex hull \cite{Oudet2004}. Alternative formulations characterized convex sets as intersections of half-spaces \cite{Lachand-Robert2006Jul}, which guarantee convexity but limit the representation of smooth geometries.

A widely adopted paradigm relies on functional parametrization using support or gauge functions \cite{Oudet2013Mar}. These approaches enable compact representations and have been successfully applied in various contexts \cite{Antunes2022May,Bayen2012Dec,Bogosel2023Feb2,Bogosel2024Nov,Ftouhi2025Jun}. However, convexity is encoded through second-order differential inequalities, leading to constrained optimization problems that are difficult to handle, especially in higher dimensions. While exact enforcement is possible in two dimensions \cite{Bayen2012Dec,Bogosel2023Feb2}, extensions to three dimensions always rely on relaxations that may compromise convexity \cite{Antunes2022May} \cite{lamberg2001numerical}.

Geometric approaches provide an alternative by directly constraining admissible deformations. For instance, \cite{Bartels2020Apr} proposes a triangulation-based framework ensuring convexity preservation through restrictions on the deformation field, while \cite{chakib2024improved} makes use of Minkowski deformations.

\subsection{Neural Representations of Convex Sets}

In the last years, a significant amount of work has been devoted to interfacing or replacing classical topology optimization methods with neural networks, parametrizing the shape using a neural network or making use of Physics Informed Neural Networks (PINNs) to solve the underlying state equation (see e.g. \cite{shin2023topology} for a review). Recent advances in machine learning have introduced flexible neural representations of shapes with built-in properties. For instance, \cite{BelieresFrendo2025Apr} enforces a volume constraint using SympNets; convexity can be enforced at the level of level-set functions \cite{Martinet2025May} using for instance input-convex neural networks (ICNNs) \cite{Amos2017ICNN}.

Beyond shape optimization, convex shape priors have been used as an inductive bias in image segmentation \cite{liu2025convex} or for convex decomposition through diffuse interface formulations such as phase-field models \cite{Deng2019Sep}. In \cite{tvetkova2025convex}, it is shown that convexity is an important property of the latent representation of decision regions.

\subsection{Main contributions}

Despite these advances, classical approaches faces significant challenges, like exact enforcement of constraints (especially in three dimensions) or flexibility when it comes to the computation of shape-related quantities. On the other hand, neural approaches often rely on implicit parametrization that makes the computation of boundary quantities harder.

In contrast, our approach aims to combine the strengths of these paradigms by providing a representation that:
(i) ensures exact convexity without constraints,
(ii) can represent any convex set,
(iii) allows for simple and precise computation of higher-order geometric terms
and (iv) is able to seamlessly treat different dimensions. We demonstrate these claims on various inverse and shape optimization problems, like shape reconstruction or the maximization of the torsional rigidity of a rod.

From a machine learning perspective, our approach can be interpreted as defining a hypothesis class of neural networks that encode convexity exactly by construction. In contrast to ICNNs, which enforce convexity of scalar functions with respect to inputs, our architecture directly parameterizes convex bodies. This provides an inductive bias tailored to geometric learning problems, enabling both expressive representations and efficient computation of shape-dependent quantities.

\section{Parametrization of convex sets using neural networks}

\begin{figure}
    \centering
    \includegraphics[width=0.85\linewidth]{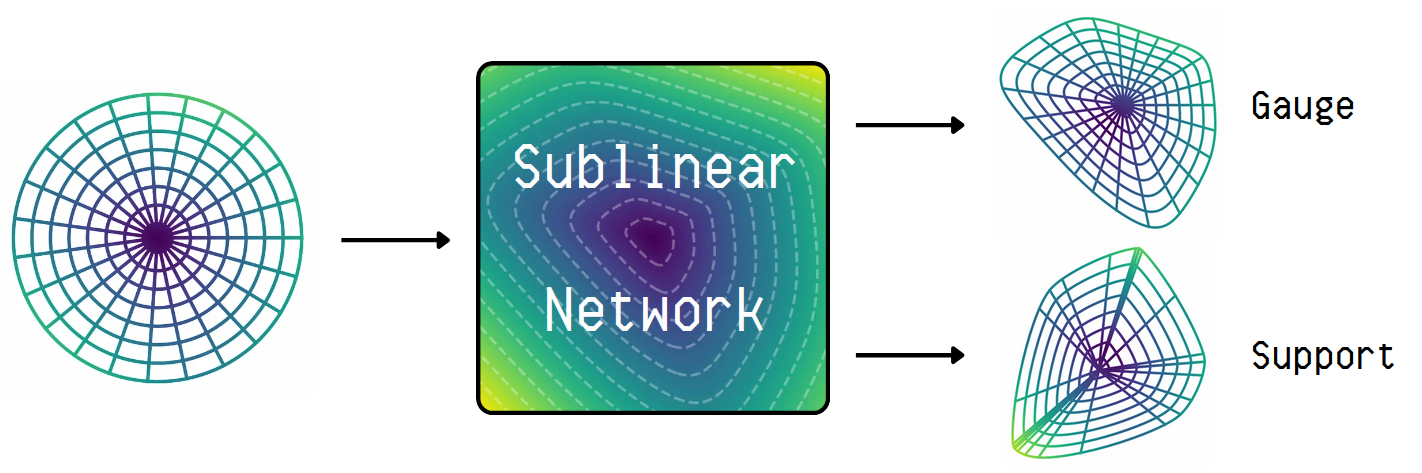}
    \caption{A sublinear network can represent both the gauge and support functions of convex sets, and allow to parameterize smooth bijections that sends the ball on convex sets.}
    \label{fig:illustration}
\end{figure}

\subsection{Gauge and support functions}

In what follows, $\mathcal{K}$ represents the set of convex bodies in $\R^d$ (i.e. compact convex sets with non-empty interior \cite{schneider}) which contains 0 in their interior. Commonly associated with a convex set $\Om \in \mathcal K$ are the gauge and support functions, respectively defined as
\[
    g_\Om(u) := \inf \left\{\lambda \geq 0 : u \in \lambda \Om \right\} 
    \qquad \text{and} \qquad
    h_\Om(u) := \sup_{x \in \Om} x \cdot u
\]
for $u\in\R^d$. It can readily be seen that both of these functions are \textit{sublinear}, i.e. subadditive and positively homogeneous. Reciprocally, if $f : \R^d \to \R$ is a sublinear function, then it is respectively the gauge and support function of the convex sets (see \cite[Theorem 1.7.1]{schneider})
\begin{equation}
    \label{eq:convex_from_func}
    \left\{ x \in \R^d : f(x) \leq 1 \right\}
    \qquad \text{ and } \qquad
    \left\{ x \in \R^d : x \cdot y \leq f(y) \text{ for all } y \in \R^d \right\}
\end{equation}

We can use the gauge function of a convex $\Om$ to give an explicit parametrization of $\Om$ as the image of the unit ball $B \subset \R^d$ in the following way:
\begin{proposition} 
    \label{prop:gauge}
    Let $g$ be a positive and sublinear function. The function $\phi:x\longmapsto \frac{\|x\|}{g(x)}x$ (extended by $0$ at the origin) is a homeomorphism from $B$ to $\phi(B)$ and the image set $\phi(B)$ is convex with gauge function $g$. 
\end{proposition}

\begin{proof}
    Let $\Om := \left\{ y \in \R^d : g(y) \leq 1 \right\}$. For $x \in B$, we have $g(\phi(x)) = \frac{\|x\|}{g(x)} g(x) = \|x\| \leq 1$. Hence, $\phi(B) \subset \Om$. On the other hand, if $y \in K$, then $x := \phi^{-1}(y)=\frac{g(y)}{\|y\|}y \in B$ and $y = \phi(x)$. Therefore, $\phi(B)=\Om$ is convex and $\phi$ is a homeomorphism. 
\end{proof}

Using \cite[Corollary 1.7.3]{schneider}, we know that if $\Om \in \mathcal K$ is such that $h_\Om$ is differentiable on $\R^d \setminus \{0\}$, then $\nabla h_\Om : \S^{n-1} \to \partial \Om$ is a homeomorphism. This leads to the following proposition:
\begin{proposition}
    \label{prop:support}
    Let $h$ be a positive and sublinear function, differentiable on $\R^d \setminus \{0\}$. The function $\phi:x\longmapsto \|x\|\nabla h(x)$ (extended by $0$ at the origin) is bijective from $B$ to $\phi(B)$ and the image set $\phi(B)$ is convex with support function $h$. 
\end{proposition}
\begin{proof}
    According to the assumptions, there exists $\Om \in \mathcal{K}$ such that $h = h_\Om$. Let $y \in \tilde \Om := \phi(B)$. Then there exists $x \in B$ such that $y= \|x\| \nabla h_\Om(x) \in \text{Conv}(\dOm \cup \{0\})$; indeed, since $0 = \phi(0) \in \tilde \Om$ and $\|x\| \leq 1$, y is a convex combination of $0$ and $\nabla h_\Om(x) \in \partial \Om$ (here, we used that $\nabla h_\Om$ is $0$-homogeneous). Since $K$ is a compact convex set, $\text{Conv}(\partial K \cup \{0\}) = \Om$ and hence $\tilde \Om \subset \Om$. On the other hand, let $y \in \Om$. Since $0$ is in the interior for $\Om$, there exists $\lambda \geq 1$ such that $\tilde y := \lambda y \in \partial \Om$. Hence, there exists $\tilde x \in \partial B$ s.t. $\tilde y = \nabla h_\Om(\tilde x)$. By taking $x = \frac{1}{\lambda} \tilde x \in B$, we have $y = \|x\| \nabla h(x)$. Hence $\tilde \Om = \Om$, meaning that $\phi(B)$ is convex.
    
    One must now prove the injectivity of $\phi$. Let $x_1, x_2 \in \R^d \setminus \{0\}$ be such that $\phi(x_1) = \phi(x_2)$. Hence, $\nabla h(x_1)$ and $\nabla h(x_2)$ are positively colinear. Moreover, $\nabla h(x_1), \nabla h(x_2) \in \partial \Om$. However, since $h>0$, we have that $0 \in \textit{int}(\Om)$. Hence, any half line originating at $0$ must intersect $\partial \Om$ exactly once (by convexity), implying that $\nabla h(x_1) = \nabla h(x_2)$ and further $\|x_1\| = \|x_2\|$. Using that $\nabla h : \S^{n-1} \to \partial \Om$ is a homeomorphism, this means that $\frac{x_1}{\|x_1\|} = \frac{x_2}{\|x_2\|}$ and hence $x_1 = x_2$.
\end{proof}

\subsection{Sublinear neural networks}

In what follows, we parametrize convex sets \textit{via} their gauge or support function. To this end, we will construct a neural network $p_\theta : \R^d \to \R$, with parameters $\theta$, that is sublinear by design.

A natural starting point is the MaxOut layer \cite{Goodfellow2013Feb}, defined as $p_\theta(x) = \max_{1 \leq i \leq N} (w_i \cdot x)$, where $w_i \in \R^d$. This choice is motivated by the classical characterization of sublinear functions as point-wise supremum of linear forms. However, this representation is non-smooth, which can be limiting in applications requiring differential quantities such as normals or curvature, or even for the application of \cref{prop:support}. To address this, we could consider to replace the maximum by a smooth approximation, like the LogSumExp (LSE). Unfortunately, the obtained function would not be sublinear anymore. Nevertheless, we are able to recover sublinearity while preserving the trace on the sphere as we will see hereafter. This requires classical tools from convex analysis, that we recall in \cref{seq:convex_analysis} for completeness.

\begin{proposition}
    \label{prop:sublinear_extension}
    Let $f : \R^d \to (-\infty, +\infty]$ be a proper convex and closed function such that $f^* \leq 0$ on its domain. Let $x \in \R^d$ and define
    \[
        h(x) := \|x\|  f\left(\frac{x}{\|x\|}\right)
    \]
    with $h(0) := 0$. Then $h$ is sublinear.
\end{proposition}

\begin{proof}
    First, $h$ is positively homogeneous by construction. Hence, sublinearity is equivalent to convexity, and we will show the latter.    
    According \cref{prop:bi-conjugate}, we can write
    \[
        f(x) = f^{**}(x) = \sup_{y \in \text{dom}(f^*)} \left\{ y \cdot x - f^*(y) \right\},
    \]
    which leads to $h(x) = \sup_{y \in \text{dom}(f^*)} \left\{ y \cdot x - \|x\| f^*(y) \right\}$. Since $f^* \leq 0$ on its domain, the function $x \mapsto  y \cdot x - \|x\| f(y)$ is convex, implying that $h$ is convex as the supremum of a family of convex functions.
\end{proof}

Using the previous proposition, we can finally define our Neural Network as
\begin{equation}
    \label{eq:lse_net}
    p_\theta(x) := \beta \|x\| \text{LSE}\left(W^T \frac{x}{\|x\|}\right)
    \qquad
    \text{where}
    \qquad
    \text{LSE}(y_1,\dots, y_n) = \log\left(\sum_{i=1}^n e^{y_i}\right).
\end{equation}

\begin{proposition}
    $p_\theta$ as defined in \cref{eq:lse_net} is a sublinear function which is $C^\infty$ on $\R^d \setminus \{0\}$.
\end{proposition}

\begin{proof}
    The regularity assumption is obvious. Since the $\text{LSE}$ is a convex function with values in $\R$, it is proper and closed. The sublinearity can be shown as follows: let $f(x) = \text{LSE}(W^T x)$. According to \cref{prop:sublinear_extension}, it is enough to show that $f^*$ is non positive. According to \cref{prop:entropy} and \cref{prop:composition}, we have that
    \[
        f^*(y) \leq W \triangleright (-S) (y) = \inf_{Wx = y} -S(x) \leq 0
    \]
    where $S$ is the entropy function defined in \cref{prop:entropy}.
\end{proof}

\begin{remark}
    According to the previous derivations, one does not need to restrict to the LSE activation function. Indeed, any convex function such that $f^* \leq 0$ on its domain would also give rise to a sublinear network. The proof of the universal approximation property (given hereafter) however heavily depends on the fact that the LSE approximates the maximum function. 
\end{remark}

\begin{remark}
    One might ask whether this construction can be extended to deeper networks. However, such a generalization is not straightforward, due to the complex interaction between convex conjugation and function composition.
\end{remark}

\subsection{Universal approximation}

According to \cref{prop:gauge} and \cref{prop:support}, the neural networks
\begin{equation}
    \label{eq:gauge_nn}
    \phi_\theta(x) := \frac{\|x\|}{p_\theta(x)} x
\end{equation}
and
\begin{equation}
    \label{eq:support_nn}
    \phi_\theta(x) := \|x\| \nabla p_\theta(x)
\end{equation}
define convex sets by $\Omega_\theta := \phi_\theta(B)$ where $p_\theta$ is respectively the gauge or support function (see \cref{fig:illustration} for an illustration of the action of these maps on the unit ball). However, it is of core interest to know whether every convex set can be represented using this architecture. It turns out that we have the following universal approximation property:
\begin{theorem}
    \label{thm:uat}
    Let $\mathcal{K}^\text{NN} := \left\{ \Om_\theta \subset \R^d : \beta > 0, W \in \R^{d\times m}, m \in \N \right\}$ be the set of convex set that can be represented by \cref{eq:gauge_nn} or \cref{eq:support_nn}. Then $\mathcal{K}^\text{NN}$ is dense in $\mathcal{K}$ with respect to the Hausdorff distance.
\end{theorem}

\begin{proof}[Sketch of proof]
    Let $\Om \subset \mathcal K$. We can approximate this set, in the Hausdorff sense, by a polyhedron $ P := \bigcap_{1\leq i\leq m} \left\{x \in \R^d : w_i \cdot x \leq 1 \right\}$. The gauge function of this polyhedron is $g_P(x) = \max_{1 \leq i \leq m} w_i \cdot x$, which can be approximated by a sublinear network $p_\theta$ using the property of the LSE to approximate the maximum function. The corresponding $\Om_\theta$ can be made as close to $P$ as desired, hence close to $\Om$.

    The case of the parametrization by a support function is done in a similar way, by observing that the support function of a polytope $P = \text{Conv}(p_1, \dots, p_n)$ is $h_P(x) = \max\{p_1 \cdot x, \dots, p_n \cdot x\}$.

    The full proof is provided in \cref{sec:proofs}.
\end{proof}

\paragraph{Symmetries}

A natural question is to know whether we can enforce symmetries on the parametrized shapes. In term of group actions, it amounts at asking if, for a certain group of symmetries $G$ we can make $\Om$ \textit{invariant} with respect to the action of $G$, i.e. $g.\Om = \Om$ for all $g \in G$. In our case, we achieve invariance by \textit{frame averaging} \cite{puny2021frame}. Given $G$ a finite subgroup of isometries of $\R^d$ and $p_\theta$ a sublinear network, we define $p_\theta^G(x) := \frac{1}{|G|} \sum_{g \in G} p_\theta(x)$. we have the following proposition:

\begin{proposition}
    \label{prop:symmetries}
    Let $p_\theta^G$ be either the gauge or support function of a convex set $\Om_\theta$. Then $\Om_\theta$ is invariant with respect to $G$.
\end{proposition}
The proof of this proposition is provided in \cref{sec:proofs}. An experiment making use of the symmetries of both the gauge and support functions can be found in \cref{sec:mahler}.

\section{Computation of shape quantities}

Shape optimization problems, as well as certain inverse problems, require to optimize a certain criteria involving quantities related to the shape, like the volume, perimeter, curvature, etc, for which we can benefit from automatic differentiation. Note that various other quantities are studied in \cite{martinet2026}.

\subsection{Integral quantities}\label{ss:integral_quantities}

We pull back the computations to the reference domain, using change of variables \cite{evans2025measure}:
\[
    \int_{\Om_\theta} f dx
    = \int_B (f \circ \phi_\theta) \Jac(\phi_\theta) dx
    \quad \mbox{ and } \quad
    \int_{\partial \Om_\theta} g d\sigma
    = \int_{\partial B}(g \circ \phi_\theta) \Jac_{\partial B}(\phi_\theta)d\sigma,
\]
where $\Jac(\phi_\theta) = |\det(D\phi_\theta)|$, $\Jac_{\partial B}(\phi_\theta) = \Jac(\phi_\theta)\|\left(D\phi_\theta\right)^{-T} n_B\|$ and $n_B(x) := x$ on $\partial B$ is the unit outward normal vector. For instance, we can compute the volume and perimeter (i.e. surface area) of $\Om_\theta$ as $\Vol(\Om_\theta)= \int_B \Jac \phi_\theta dx$ and $\Per(\Om_\theta) = \int_{\partial B} \Jac_{\partial B}(\phi_\theta)d\sigma$. All the differential quantities are automatically computed using PyTorch \cite{Paszke2019Dec}. The integrals can be discretized using either Monte-Carlo or fixed quadrature points. We chose the latter approach, as we will mainly use L--BFGS as an optimizer, which is known for being sensitive to noise. In particular, we use a Fibonacci lattice approach for the discretizations on the $2$-sphere \cite{Gonzalez2010Jan}.

\subsection{Geometric-differential quantities}


Let $y = \phi_\theta(x)$, $x \in \partial B$. We can express the normal vector at $y \in \partial \Om_\theta$ by 
\[
    n_\theta\left(y \right) = \frac{\left(D\phi_\theta\right)^{-T}(x) n_B(x)}{\left\|\left(D\phi_\theta\right)^{-T}(x) n_B(x)\right\|}.
\]
This naturally defines an extended vector field in $\R^d \setminus \{0\}$. The mean curvature and Gaussian curvature are respectively defined as 
\[
    H_\theta(y) = (d-1)^{-1} \tr S_y
    \qquad \text{and} \qquad
    \kappa_\theta(y) = \det{S_y}
\]
where $S_y$ is the \textit{Weingarten map} at $y$. For more details, see \cref{sec:geo_diff}.

\subsection{PDE-related quantities}

It is common in shape optimization to consider quantities that depends on the solution of a PDE. For instance, optimal design of structures often minimize for the compliance, which is a by-product of the linear elasticity equation; aerodynamic shape optimization needs to solve Navier-Stokes. In this section, we show how it is possible to seamlessly bridge precise and robust mesh free methods with the auto-differentiation of PyTorch in order to easily compute derivatives of PDE-dependent quantities in the simple case of the Poisson equation.

\paragraph{Mesh free Galerkin method:} For $f \in L^2(\Om_\theta)$, the Poisson problem with Dirichlet boundary condition aims at finding $u \in H^1_0(\Om_\theta)$ which satisfies
\begin{equation}
    \label{eq:poisson}
    \begin{cases}
        -\Delta u &=\ \  f\ \ \ \mbox{ in } \Om_\theta,\\
        \ \ \ \ \ u &=\ \  0\ \ \   \mbox{ on } \dOm_\theta.
    \end{cases}
\end{equation}
Passing to the \textit{weak formulation} \cite{evans2022partial} and changing variables, \cref{eq:poisson} can be express in a weak sense as
\[
    \int_{B} A_\theta \nabla u \cdot \nabla v = \int_B (\Jac \phi_\theta) (f \circ \phi_\theta) v,
\]
for $u,v \in H^1_0(B)$, where $A_\theta := (\Jac \phi_\theta) (D\phi_\theta)^{-1} (D\phi_\theta)^{-T}$. The Dirichlet boundary condition is weakly enforced by adding a penalization term with penalty parameter $\alpha>0$. The resulting problem is then discretized on a subspace spanned by radial basis functions (RBFs) $\varphi_i(x) := \psi(|x - x_i|)$, $x_i \in B$, possibly augmented by a polynomial basis \cite{Wendland} leading to a system $K \bar u = \bar f$, where 
\[
    K_{ij} = \int_{B} A_\theta \nabla \varphi_i \cdot \nabla \varphi_j  + \alpha \int_{\partial B} \varphi_i \varphi_j  
    \qquad \text{and} \qquad 
    \bar f_i = \int_B (\Jac \phi_\theta) (f \circ \phi_\theta) \varphi_i.
\]
The integrals are discretized as previously described. We can then solve the previous linear system and recover the solution $u(x) = \sum_i \bar u_i \varphi(x)$.

\begin{remark}
    The RBF-Galerkin method is chosen because of its theoretical convergence guarantees (unlike other meshless methods relying on the strong formulation of the PDE like the Kansa method \cite{fasshauer2007meshfree} or PINN-based methods \cite{dissanayake1994neural}) and is easy to implement in PyTorch.
\end{remark}


\paragraph{Method of fundamental solutions} An important particular case of \cref{eq:poisson} is for $f = 1$, in which case the equation can be solved thanks to a method exhibiting spectral convergence. The details are provided in \cref{seq:mfs_galerkin}.

\paragraph{A note on the use of PINNs}

One may wonder why we do not use PINNs to solve the PDEs, as in related work \cite{BelieresFrendo2025Apr}. Our choice is guided by both practical and structural considerations. On the practical side, recent studies indicate that PINNs do not match the efficiency of classical solvers on several low-dimensional PDEs \cite{grossmann2024can,mcgreivy2024weak}, which is confirmed by our experiments. On the structural side, while \cite{BelieresFrendo2025Apr} relies on a min--min formulation enabling joint optimization, most shape optimization problems are naturally min--max, requiring the inner problem to be (approximately) solved at each step, which increases computational cost.

For completeness, we include a comparison with PINNs in \cref{seq:pinn}, where classical methods remain way more efficient in our setting.

\paragraph{On shape derivatives and the FEM}

Classical approaches to minimizing a shape functional $J$ rely on shape derivatives evaluated on a mesh. Their numerical implementation typically involves additional components such as adjoint states and extension–regularization procedures, which require solving PDEs on the domain. In contrast, our approach avoids these steps by leveraging automatic differentiation.

That said, mesh-based methods remain highly accurate and may be desirable in this context. However, there is currently no standard PyTorch-based framework for finite element computations. We discuss a possible interface with external FEM libraries in \cref{seq:fem}.

\subsection{Which representation should be used?} 

In most cases, the gauge and support functions representations are interchangeable, since computations of shape quantities only rely on the fact that $\phi_\theta$ is a smooth bijection and not on its particular structure. However, there is certain cases where one of the parametrization is preferable, like the gauge parametrization in \cref{subseq:fit} or the support parametrization in \cref{subseq:minkowski}. The gauge parametrization may also be preferred when one needs an explicit inverse, as illustrated in \cref{seq:fem}. Moreover, certain geometric quantities that we do not discuss here are easier to compute with a certain representation (for instance, the \textit{width} is easily computed in the support function parametrization).

\begin{remark}
Notably, the arguments in this section do not rely on the convexity of $\Omega_\theta$ or on any special structure of $\phi_\theta$. This suggests that the framework can be extended beyond the present setting to arbitrary invertible neural networks $\phi_\theta$, which we leave to future work.
\end{remark}

\begin{figure*}[t]
\centering

\setlength{\tabcolsep}{2pt}

\begin{tabular}{c c c c c c}

& & $\sigma=0.0$ & $\sigma=0.01$ & $\sigma=0.05$ & $\sigma=0.1$ \\

&
\quad Reference   \quad \rotatebox{90}{\parbox{2cm}{\centering Samples}} &
\includegraphics[width=0.13\linewidth]{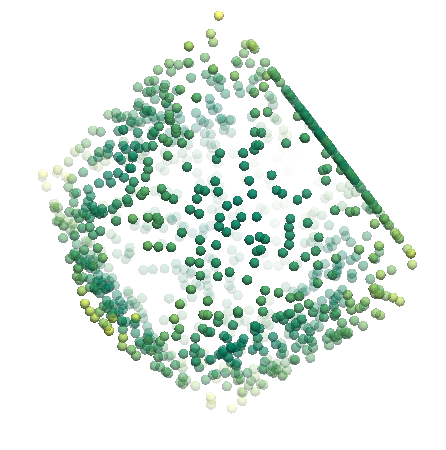} &
\includegraphics[width=0.13\linewidth]{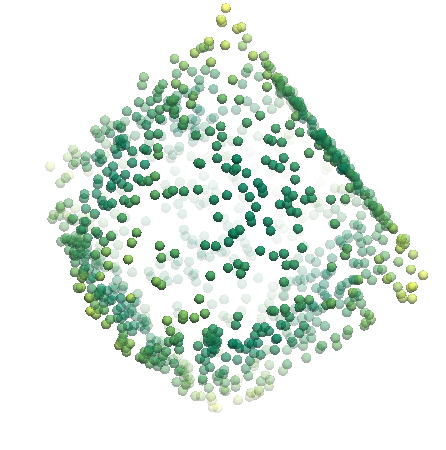} &
\includegraphics[width=0.13\linewidth]{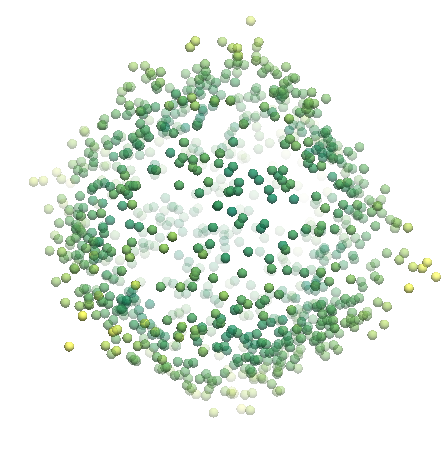} &
\includegraphics[width=0.13\linewidth]{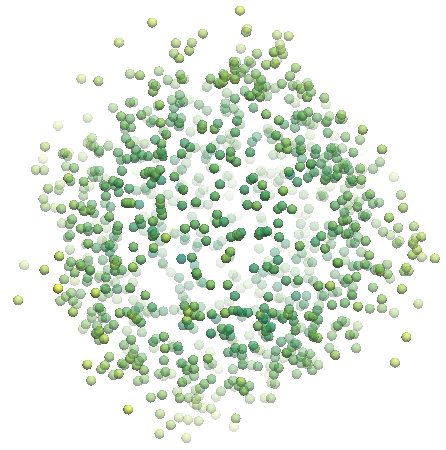} \\

\rotatebox{90}{\parbox{2cm}{\centering Octahedron}} &
\includegraphics[width=0.13\linewidth]{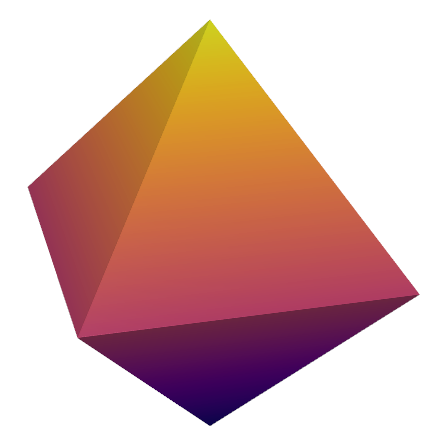} &
\includegraphics[width=0.13\linewidth]{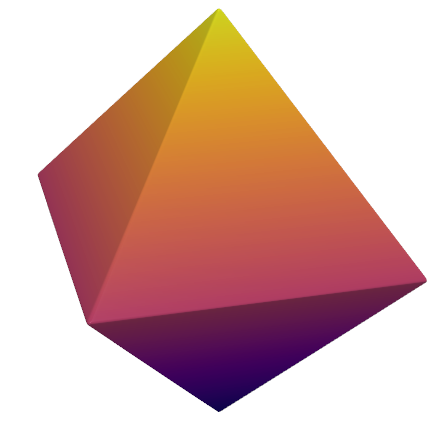} &
\includegraphics[width=0.13\linewidth]{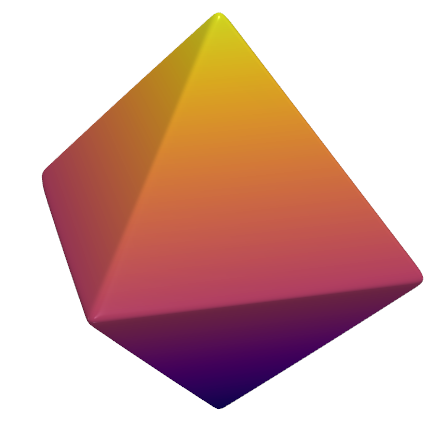} &
\includegraphics[width=0.13\linewidth]{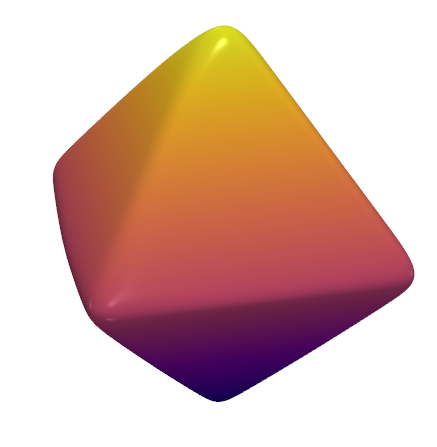} &
\includegraphics[width=0.13\linewidth]{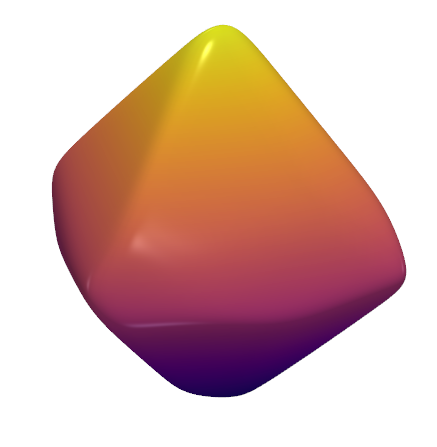} \\

\rotatebox{90}{\parbox{2cm}{\centering Ball}} &
\includegraphics[width=0.13\linewidth]{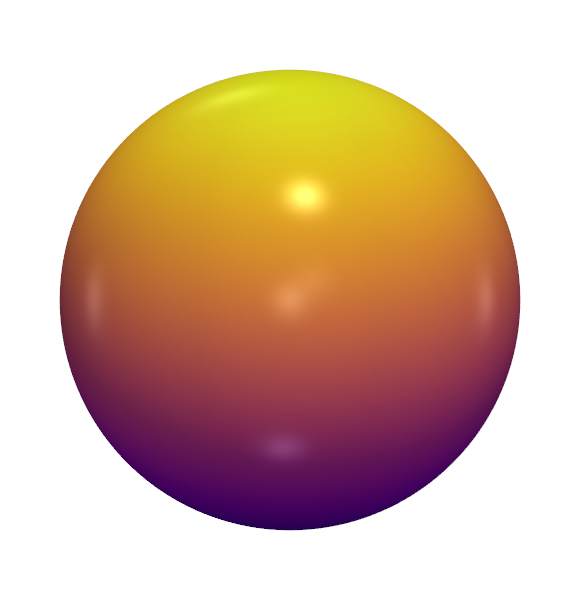} &
\includegraphics[width=0.13\linewidth]{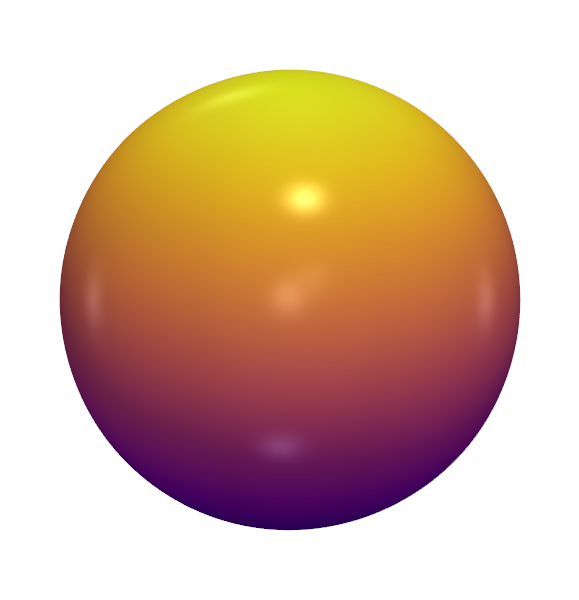} &
\includegraphics[width=0.13\linewidth]{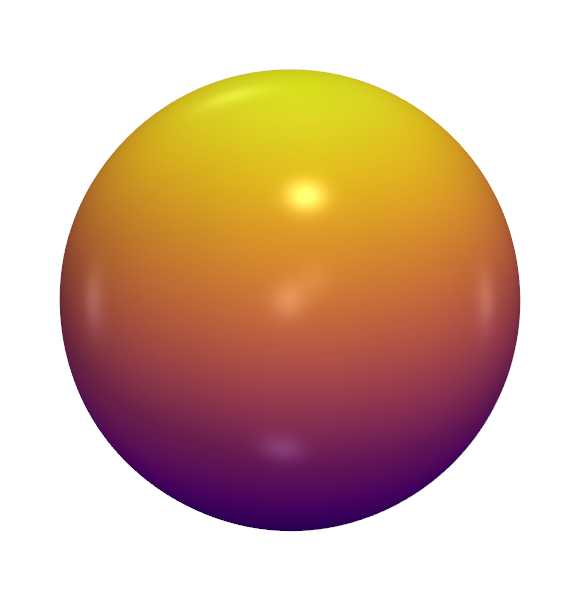} &
\includegraphics[width=0.13\linewidth]{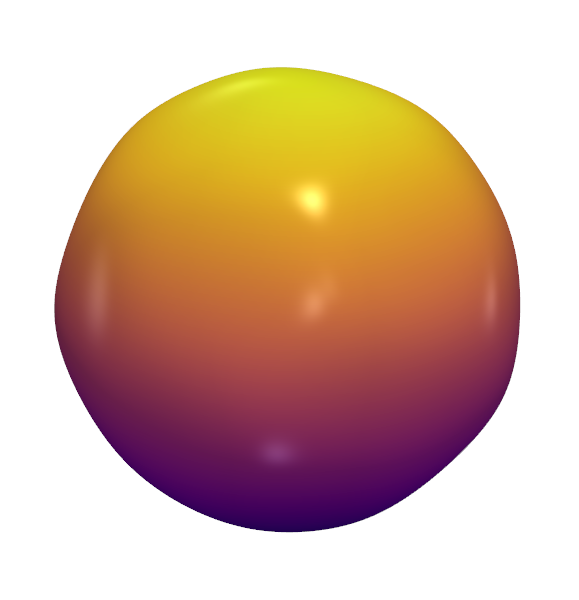} &
\includegraphics[width=0.13\linewidth]{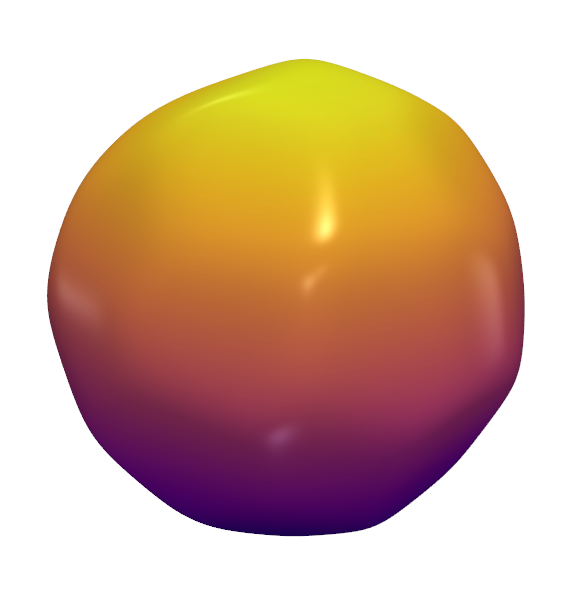} \\

\rotatebox{90}{\parbox{1.7cm}{\centering Cube}} &
\includegraphics[width=0.13\linewidth]{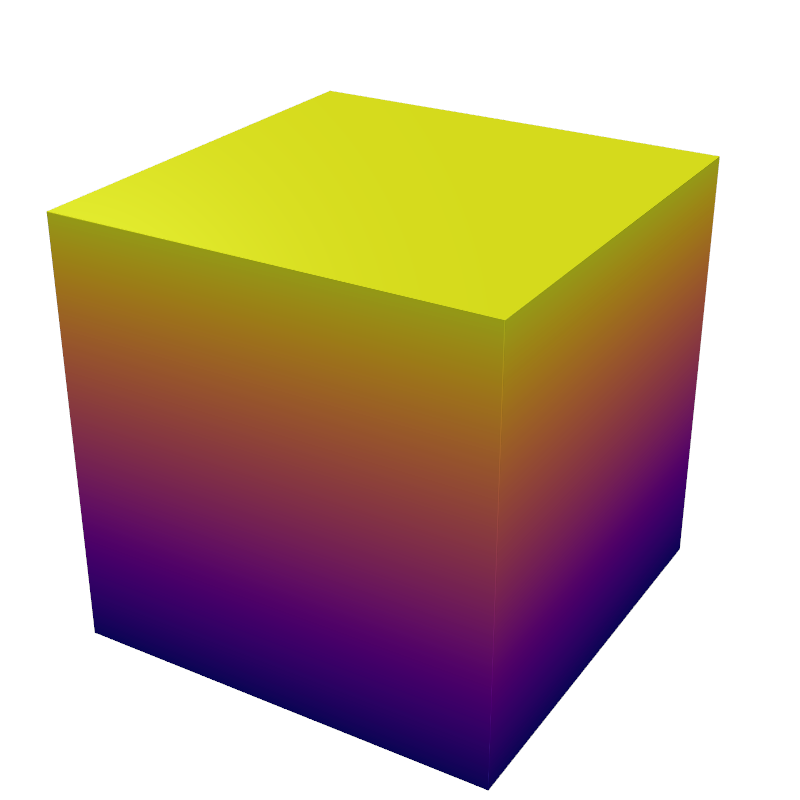} &
\includegraphics[width=0.13\linewidth]{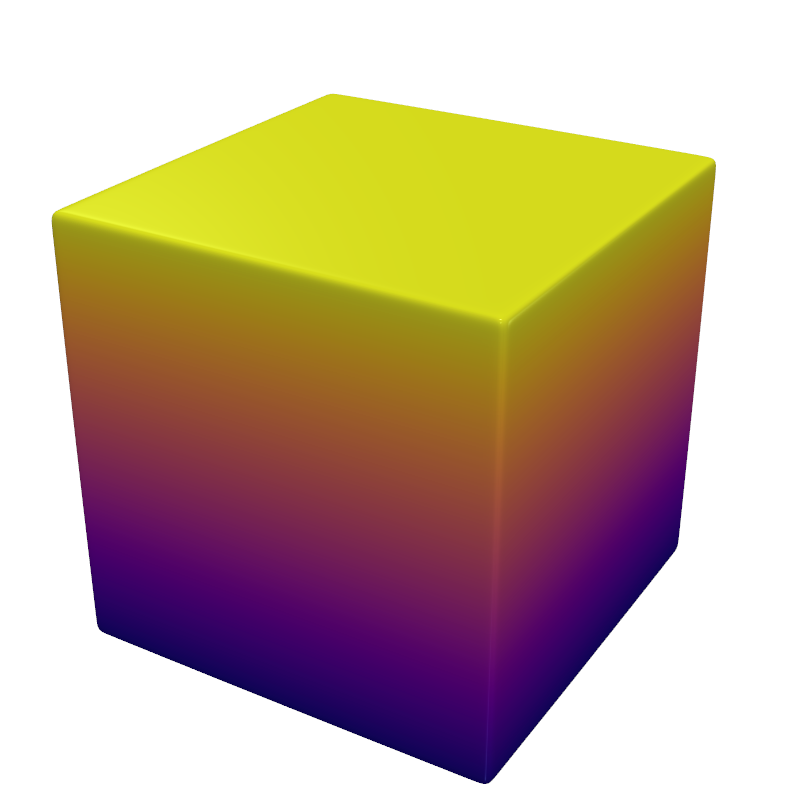} &
\includegraphics[width=0.13\linewidth]{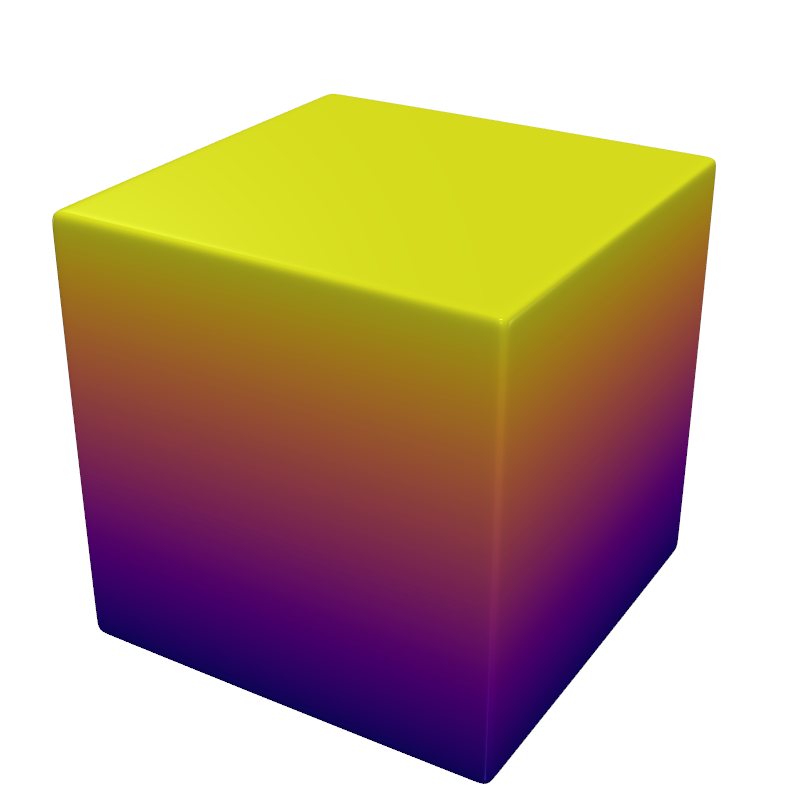} &
\includegraphics[width=0.13\linewidth]{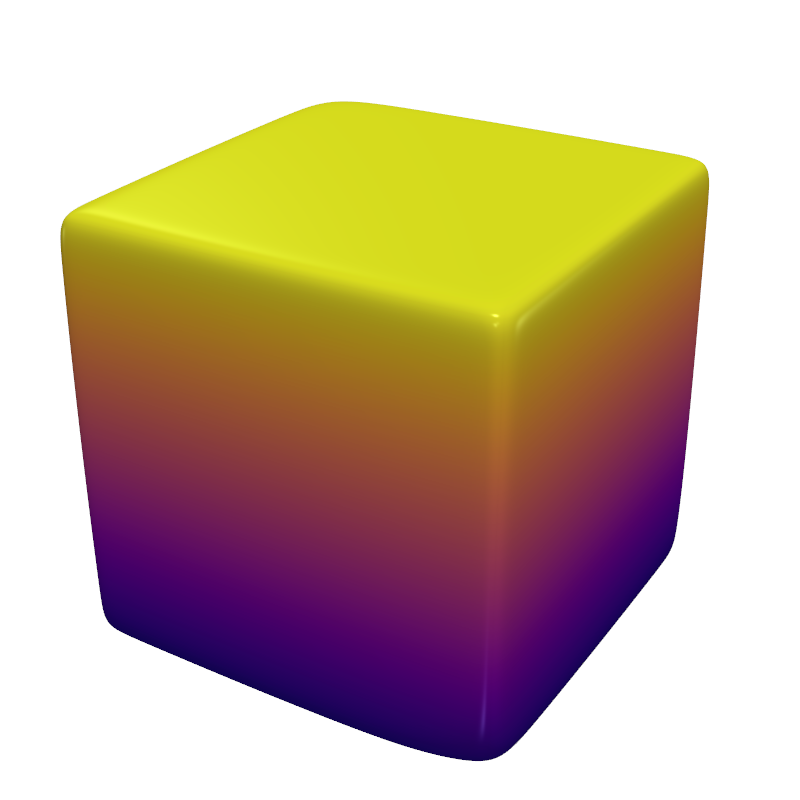} &
\includegraphics[width=0.13\linewidth]{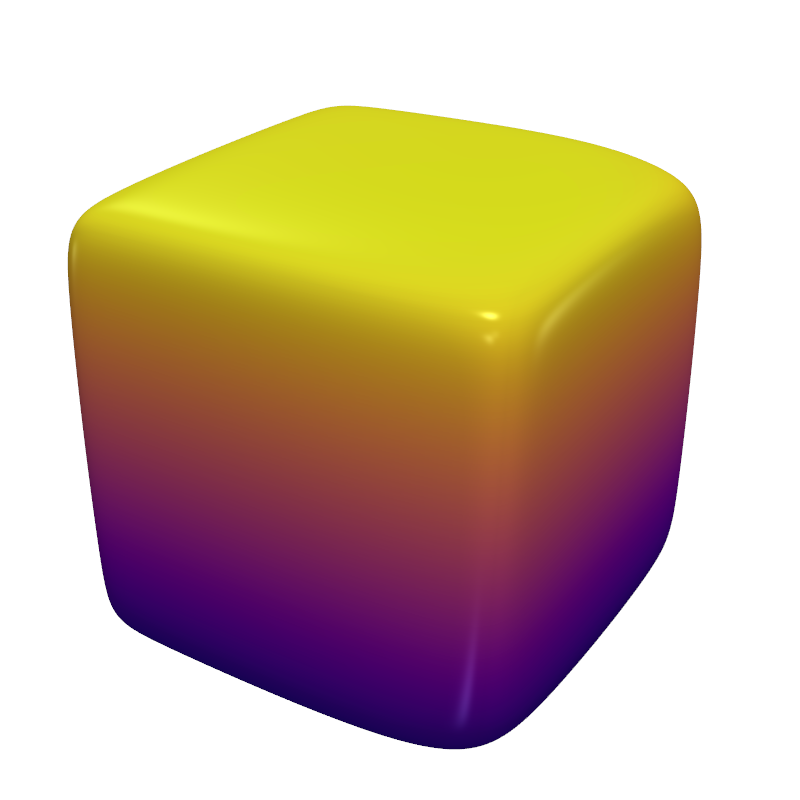} \\

\end{tabular}

\vspace{0.5em}

\caption{
Reconstruction of convex shapes from noisy point clouds.
First column: reference shapes.
Top row: input point clouds sampled from the octahedron under increasing Gaussian noise levels $\sigma$.
Remaining rows: reconstructed shapes for different target geometries.
}

\label{fig:noisy_fit}
\end{figure*}

\section{Application to shape optimization problems}

In this section, we present the performance of our method on a range of problems arising from geometry and shape optimization. For fairness and reproducibility, all experiments are conducted on a single CPU (AMD Ryzen 7 Pro) with 28 GB of RAM. The computational time for all two-dimensional cases is on the order of one minute, while in three dimensions it is typically on the order of ten minutes. The implementation is also CUDA-compatible and can therefore benefit from standard GPU acceleration. A direct quantitative comparison with classical methods is not possible in practice, as no publicly available implementations exist in a form that would allow a consistent and reproducible evaluation.

\subsection{Learning convex sets with noisy boundary samples}
\label{subseq:fit}

The first problem we consider is to reconstruct a convex shape from noisy observations. More precisely, given sampled $y_1, \dots, y_n \in \R^d$, we minimize $L(\theta) = \sum_{i=1}^n |p_\theta(y_i) - 1|^2$ where $p_\theta$ is the gauge function of the convex set $\Om_\theta$. This loss is motivated by the fact that $\dOm_\theta = \{ p_\theta = 1 \}$.

In practice, the observations $y_i$ are generated synthetically as follows: first, $x_i \sim \mathcal{U}(\partial B)$. Then, $y_i = \phi_{\text{target}}(x_i) + \varepsilon_i$ where $\varepsilon_i \sim \mathcal{N}(0, \sigma)$, and $\phi_{\text{target}}$ denotes a maps from $B$ to $\Omega_{\text{target}}$. The results, shown in \cref{fig:noisy_fit}, are obtained with $n = 1000$ samples and varying noise levels across three different shapes. We observe that the convex inductive bias enables an accurate reconstruction of the shapes, even with a relatively small number of samples and substantial noise. Additional statistical experiments assessing sensitivity to noise and to the amount of samples over multiple runs are reported in \cref{seq:fit_statistic}.

\subsection{Optimization of a Poisson problem}

One of the simplest PDE-constrained shape optimization problem is the following: considering a function $f: \R^d \to \R$ and $u_\theta$ to be the solution of \cref{eq:poisson} on a domain $\Om_\theta \in \mathcal{K}$, what is the minimum of $J(\Om_\theta) := \int_{\Om_\theta} u_\theta$? As described previously, the solution $u_\theta$ is computed by a mesh free Galerkin method, while the integral is evaluated by change of variables. We performs our experiments in the same settings as in \cite{Bogosel2023Feb2}, i.e. in dimension $2$ for two different functions $f$. The optimal shapes are given in \cref{fig:poisson_galerkin}. They can be compared to the ones in \cite{Bogosel2023Feb2}.


\begin{figure}
    \centering
    \includegraphics[width=0.3\linewidth]{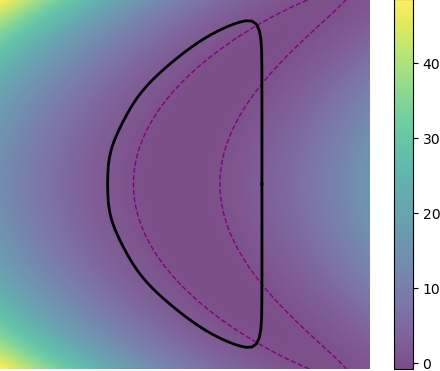}
    \quad
    \includegraphics[width=0.3\linewidth]{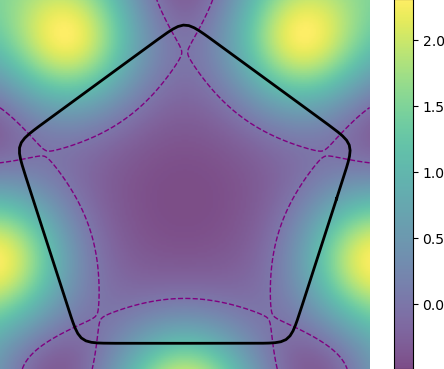}
    \caption{Optimal shape for the Poisson-based shape optimization problem (in black) along with the function $f$. The dotted line is the $0$ level set.}
    \label{fig:poisson_galerkin}
\end{figure}

\subsection{Maximization of the gradient of the torsion function}

In linear elasticity, the quantity $\|\nabla u_\Om \|_{L^\infty}$ (where $u_\Om$ is the torsion function \cref{eq:torsion}) represents the maximal shear stress of a rod with section $\Om$. In the case of a planar convex domain $\Om \subset \R^d$, many authors were interested in knowing what shape maximized this quantity when prescribing the area or the perimeter (see \cite{burdzy2025} and references therein). Using homogeneity arguments, it amounts at maximizing respectively $\frac{\|\nabla u_\Om\|_{L^\infty}}{\Vol(\Om)^{1/d}}$ and $\frac{\|\nabla u_\Om\|_{L^\infty}}{\Per(\Om)^{1/(d-1)}}$. Assuming enough regularity, the maximum principle applied on $|\nabla u|^2$ ensures that the infinity norm is attained on the boundary, i.e. $\|\nabla u_\Om\|_{L^\infty} = \max_{\partial \Om} |\partial_n u|$. For a parametrized convex set $\Om_\theta$, maximizing the previous quantities is hence equivalent to maximizing respectively
\[
    J_{\Vol}(\theta) := \frac{|\partial_n u_{\Om_\theta}(\phi_\theta(x))|}{\Vol(\Om_\theta)^{1/d}}
    \qquad \text{ and } \qquad
    J_{\Per}(\theta) := \frac{|\partial_n u_{\Om_\theta}(\phi_\theta(x))|}{\Per(\Om)^{1/(d-1)}}.
\]
for a fixed $x \in \partial B$. The optimal shapes in dimension $2$ and $3$ are given in \cref{fig:ilias}. In order to show that our algorithm easily adapts to higher dimensions, we give in \cref{tab:ilias} the optimal values obtained with our method for each functionals in dimension up to $4$. In dimension $2$, the values that are obtained matches previous numerical experiments found in the literature \cite{burdzy2025}.

\begin{figure}
    \centering
    \includegraphics[width=0.24\linewidth]{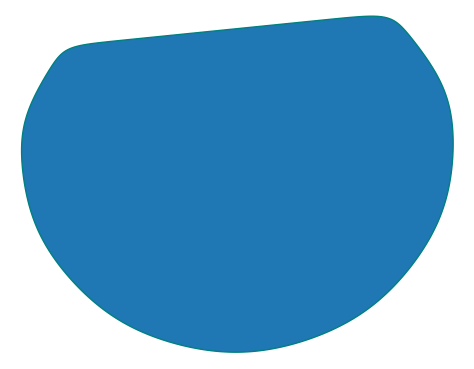}
    \includegraphics[width=0.24\linewidth]{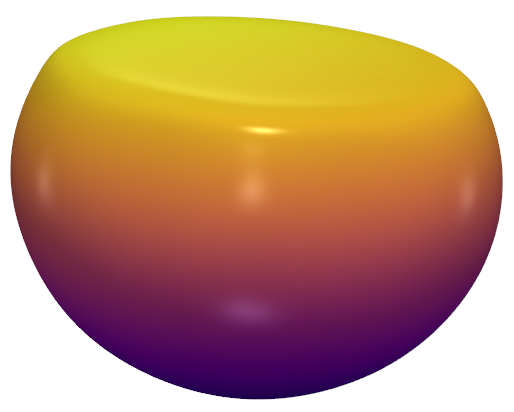}
    \includegraphics[width=0.24\linewidth]{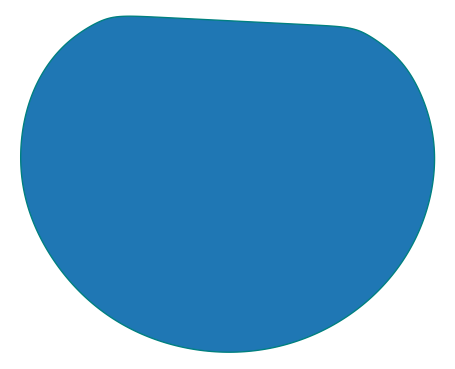}
    \includegraphics[width=0.24\linewidth]{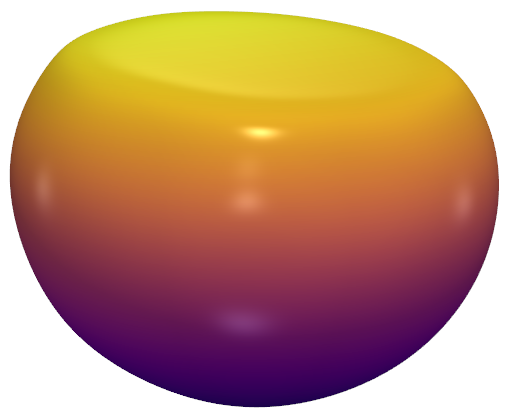}
    \caption{Optimal shapes for maximizing the torsion gradient under two constraints: fixed volume (left) and fixed perimeter (right). While the resulting shapes are similar, the choice of normalization produces distinct geometries.}
    \label{fig:ilias}
\end{figure}

\begin{table}
    \centering 
    \caption{Optimal values for $J_\Vol$ and $J_\Per$ in several dimensions.}
    \label{tab:ilias}
    \begin{tabular}{c c c} 
         \toprule
         Dimension & $J_\Vol^*$ &$J_\Per^*$ \\ [0.5ex] 
         \midrule
         2 & 0.35809 & 0.09886 \\
         3 & 0.30918 & 0.13842 \\
         4 & 0.28397 & 0.15532 \\
         \bottomrule
    \end{tabular}
\end{table}

\subsection{Minkowski problem}
\label{subseq:minkowski}

The Minkowski problem \cite{huang2025minkowski} is a foundational problem in convex and differential geometry, connected to the famous Monge--Ampère equation. Informally speaking, it is concerned with the question of existence of a convex set with prescribed Gaussian curvature. More precisely, let $g : \S^{n-1} \to \R$ be a positive, continuous function. Does there exists a convex set $\Om$ such that $\kappa_\Om \circ n_\Om^{-1} = g$ on $\partial B$? It is known that such $\Om$ exists if and only if $g$ verifies $\int_{\partial B} \frac{u}{g(u)} du = 0$. We can approximate this problem by formulating it as a regression problem, by minimizing the relative mean squared error between the actual and target curvature:
\[
    L(\theta) = \int_{\partial B} \left|\frac{\kappa_\theta(n_\theta^{-1}(u)) - g(u)}{g(u)}\right|^2 du.
\]
Except for $n_\theta^{-1}$, everything is easily computable in the present framework. However, as it is pointed out in \cite{schneider}, if $p_\theta$ is the support function of $\Om_\theta$ (i.e., choosing the parametrization \cref{eq:support_nn}) we have $n_\theta^{-1} = \nabla p_\theta = \phi_\theta$ on $\partial B$. Hence, the previous loss is readily computable. We show that our method successfully applies to the Minkowski problem in \cref{fig:minkowski}. The final $L^2$ relative error is of the order of $10^{-2}$.
\begin{figure}
    \centering
    \includegraphics[width=0.3\linewidth]{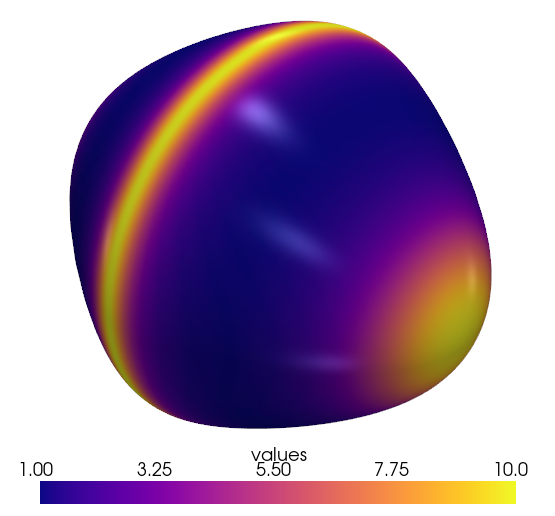}
    \hspace{1em}
    \includegraphics[width=0.3\linewidth]{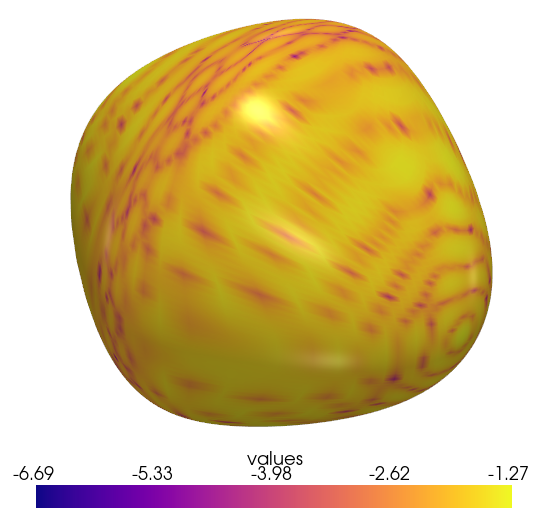}
    \caption{Solution to the Minkowski problem. Left: prescribed curvature. Right: $\log_{10}$ relative error. The method accurately recovers the target curvature distribution.}
    \label{fig:minkowski}
\end{figure}

\paragraph{Reproducibility}

The code is fully available on GitHub: \url{https://anonymous.4open.science/r/SublinearNet-5726/}.

\section{Discussion}

The proposed approach allows for a provably expressive and unconstrained representation of convex sets. Profiting from the auto-differentiation capabilities of PyTorch, it allows to easily solve a large variety of shape optimization problems without any shape derivative computation.

\paragraph{Limitations} However, several limitations remain. The reliance on quasi-Newton methods like L--BFGS with fixed discretizations limits scalability due to the curse of dimensionality. Certain geometric quantities (e.g., curvature) may become unstable near non-smooth shapes. Memory usage is also higher than in classical approaches, since PDE solvers must remain in the computational graph. Finally, symmetry enforcement increases computational cost linearly with respect to the size of the symmetry group due to repeated evaluations.

\paragraph{Broader impact}
We validate the proposed method primarily on mathematical shape optimization problems, where it may help accelerate exploration and conjecture generation. Extending the framework to more realistic engineering settings—such as structural or aerodynamic optimization—could further broaden its applicability and impact.

\section{Conclusion}

We introduced a neural parametrization of convex sets based on sublinear networks, ensuring exact convexity while retaining strong approximation capabilities. We proved universal approximation in the Hausdorff distance and showed how the parametrization enables efficient evaluation of geometric and PDE-dependent quantities via automatic differentiation.

Numerical experiments demonstrate competitive performance on a range of shape optimization problems, and show that classical solvers combined with our parametrization outperform PINN-based approaches in both accuracy and efficiency.

Future work includes extending the framework to more complex PDEs (e.g., elasticity or fluid dynamics) and designing architectures that preserve additional geometric constraints, like the volume (using measure-preserving neural networks) or the perimeter, for which a new architecture must probably be derived.

\section*{Acknowledgments}

Large language models were used to assist with code generation—particularly for plotting—and to explore mathematical ideas. 

\bibliographystyle{unsrtnat}  
\bibliography{biblio.bib}

@article{buttazzo1997shape,
  title={Shape optimization problems over classes of convex domains},
  author={Buttazzo, Giuseppe and Guasoni, Paolo},
  journal={Journal of Convex Analysis},
  volume={4},
  pages={343--352},
  year={1997},
  publisher={Heldermann Verlag}
}

@book{fasshauer2007meshfree,
  title={Meshfree approximation methods with Matlab (With Cd-rom)},
  author={Fasshauer, Gregory E},
  volume={6},
  year={2007},
  publisher={World Scientific Publishing Company}
}

@BOOK{schneider,
	title = {Convex Bodies: The Brunn-Minkowski Theory},
	publisher = {Cambridge University Press},
	author = {Schneider, R.},
	year = {2013},
	edition = {2nd expanded edition},
}

@article{Paszke2019Dec,
  title={Pytorch: An imperative style, high-performance deep learning library},
  author={Paszke, Adam and Gross, Sam and Massa, Francisco and Lerer, Adam and Bradbury, James and Chanan, Gregory and Killeen, Trevor and Lin, Zeming and Gimelshein, Natalia and Antiga, Luca and others},
  journal={Advances in neural information processing systems},
  volume={32},
  year={2019}
}

@article{Ftouhi2025Jun,
	author = {Ftouhi, Ilias},
	title = {{Improved Description of Blaschke{\textendash}Santal{\ifmmode\acute{o}\else\'{o}\fi} Diagrams via Numerical Shape Optimization}},
	journal = {Appl. Math. Optim.},
	volume = {91},
	number = {3},
	pages = {55},
	year = {2025},
	publisher = {Springer US},
}

@article{Wendland,
 author = {Holger Wendland},
 journal = {Mathematics of Computation},
 number = {228},
 pages = {1521--1531},
 publisher = {American Mathematical Society},
 title = {Meshless Galerkin Methods Using Radial Basis Functions},
 volume = {68},
 year = {1999}
}

@book {HPb,
    AUTHOR = {Henrot, Antoine and Pierre, Michel},
     TITLE = {Shape variation and optimization},
    SERIES = {EMS Tracts in Mathematics},
    VOLUME = {28},
 PUBLISHER = {European Mathematical Society (EMS), Z\"{u}rich},
      YEAR = {2018},
}

@incollection{Martinet2025May,
	author = {Martinet, Eloi and Bungert, Leon},
	title = {{Meshless Shape Optimization Using Neural Networks and Partial Differential Equations on Graphs}},
	booktitle = {{Scale Space and Variational Methods in Computer Vision}},
	journal = {SpringerLink},
	pages = {285--297},
	year = {2025},
}

@InProceedings{Deng2019Sep,
author = {Deng, B. and Genova, K. and Yazdani, S. and Bouaziz, S. and Hinton, G. and Tagliasacchi, A.},
title = {CvxNet: Learnable Convex Decomposition},
booktitle = {Proceedings of the IEEE/CVF Conference on Computer Vision and Pattern Recognition (CVPR)},
year = {2020}
}

@inproceedings{Goodfellow2013Feb,
  title={Maxout networks},
  author={Goodfellow, Ian and Warde-Farley, David and Mirza, Mehdi and Courville, Aaron and Bengio, Yoshua},
  booktitle={International conference on machine learning},
  pages={1319--1327},
  year={2013},
  organization={PMLR}
}

@book{boyd2004convex,
  title={Convex optimization},
  author={Boyd, Stephen P and Vandenberghe, Lieven},
  year={2004},
  publisher={Cambridge university press}
}

@incollection{Bauschke2017,
  title={Correction to: convex analysis and monotone operator theory in Hilbert spaces},
  author={Bauschke, Heinz H and Combettes, Patrick L},
  booktitle={Convex analysis and monotone operator theory in Hilbert spaces},
  year={2020},
  publisher={Springer}
}

@article{Antunes2022May,
	author = {Antunes, Pedro R. S. and Bogosel, Beniamin},
	title = {{Parametric shape optimization using the support function}},
	journal = {Comput. Optim. Appl.},
	volume = {82},
	number = {1},
	pages = {107--138},
	year = {2022},
	publisher = {Springer US},
}

@article{Bartels2020Apr,
	author = {Bartels, S{\ifmmode\ddot{o}\else\"{o}\fi}ren and Wachsmuth, Gerd},
	title = {{Numerical Approximation of Optimal Convex Shapes}},
	journal = {SIAM J. Sci. Comput.},
	year = {2020},
	month = apr,
	publisher = {Society for Industrial and Applied Mathematics},
}

@article{Bogosel2024Nov,
	author = {Bogosel, Beniamin and Henrot, Antoine and Michetti, Marco},
	title = {{Optimization of Neumann Eigenvalues Under Convexity and Geometric Constraints}},
	journal = {SIAM J. Math. Anal.},
	year = {2024},
	month = nov,
	publisher = {Society for Industrial and Applied Mathematics University City, Philadelphia},
}

@article{Bayen2012Dec,
	author = {Bayen, T{\ifmmode\acute{e}\else\'{e}\fi}rence and Henrion, Didier},
	title = {{Semidefinite programming for optimizing convex bodies under width constraints}},
	journal = {Optim. Methods Softw.},
	year = {2012},
	publisher = {Taylor {\&} Francis},
}

@article{BelieresFrendo2025Apr,
	author = {B{\ifmmode\acute{e}\else\'{e}\fi}li{\ifmmode\grave{e}\else\`{e}\fi}res Frendo, Amaury and Franck, Emmanuel and Michel-Dansac, Victor and Privat, Yannick},
	title = {{Volume-preserving geometric shape optimization of the Dirichlet energy using variational neural networks}},
	journal = {Neural Networks},
	volume = {184},
	pages = {106957},
	year = {2025},
	publisher = {Pergamon},
}

@article{Bogosel2023Feb2,
	author = {Bogosel, Beniamin},
	title = {{Numerical Shape Optimization Among Convex Sets}},
	journal = {Appl. Math. Optim.},
	volume = {87},
	number = {1},
	pages = {1},
	year = {2023},
	publisher = {Springer US},
}

@article{Gonzalez2010Jan,
	author = {Gonz{\ifmmode\acute{a}\else\'{a}\fi}lez, {\ifmmode\acute{A}\else\'{A}\fi}lvaro},
	title = {{Measurement of Areas on a Sphere Using Fibonacci and Latitude{\textendash}Longitude Lattices}},
	journal = {Math. Geosci.},
	volume = {42},
	number = {1},
	pages = {49--64},
	year = {2010},
	publisher = {Springer-Verlag},
}

@article{Bogosel2016Nov,
	author = {Bogosel, Beniamin},
	title = {{The method of fundamental solutions applied to boundary eigenvalue problems}},
	journal = {J. Comput. Appl. Math.},
	volume = {306},
	pages = {265--285},
	year = {2016},
	publisher = {North-Holland},
}

@article{Oudet2004,
	author = {Oudet, {\ifmmode\acute{E}\else\'{E}\fi}douard},
	title = {{Numerical minimization of eigenmodes of a membrane with respect to the domain}},
	journal = {ESAIM Control Optim. Calc. Var.},
	volume = {10},
	number = {3},
	pages = {315--330},
	year = {2004},
}

@article{Lachand-Robert2006Jul,
	author = {Lachand-Robert, Thomas and Oudet, {\ifmmode\acute{E}\else\'{E}\fi}douard},
	title = {{Minimizing within Convex Bodies Using a Convex Hull Method}},
	journal = {SIAM J. Optim.},
	year = {2006},
}

@article{Oudet2013Mar,
	author = {Oudet, {\ifmmode\acute{E}\else\'{E}\fi}douard},
	title = {{Shape Optimization Under Width Constraint}},
	journal = {Discrete Comput. Geom.},
	volume = {49},
	number = {2},
	pages = {411--428},
	year = {2013},
}

@inproceedings{Amos2017ICNN,
  title={Input Convex Neural Networks},
  author={Amos, Brandon and Xu, Lei and Kolter, J. Zico},
  booktitle={International Conference on Machine Learning (ICML)},
  year={2017}
}

@article{liu2025convex,
  title={Convex Shape Prior for Deep Convolution Neural Network-Based Image Segmentation},
  author={Liu, Jun and Zhang, Kehui and Tai, Xue-Cheng and Luo, Shousheng},
  journal={Journal of Mathematical Imaging and Vision},
  volume={67},
  number={6},
  pages={61},
  year={2025},
}

@article{tvetkova2025convex,
  title={On convex decision regions in deep network representations},
  author={T{\v{e}}tkov{\'a}, Lenka and Br{\"u}sch, Thea and Dorszewski, Teresa and Mager, Fabian Martin and Aagaard, Rasmus {\O}rtoft and Foldager, Jonathan and Alstr{\o}m, Tommy Sonne and Hansen, Lars Kai},
  journal={Nature Communications},
  volume={16},
  number={1},
  pages={5419},
  year={2025},
}

@article{chakib2024improved,
  title={An improved numerical approach for solving shape optimization problems on convex domains},
  author={Chakib, Abdelkrim and Khalil, Ibrahim and Sadik, Azeddine},
  journal={Numerical Algorithms},
  volume={96},
  number={2},
  pages={621--663},
  year={2024},
}

@article{shin2023topology,
  title={Topology optimization via machine learning and deep learning: a review},
  author={Shin, Seungyeon and Shin, Dongju and Kang, Namwoo},
  journal={Journal of Computational Design and Engineering},
  volume={10},
  number={4},
  pages={1736--1766},
  year={2023},
}

@article{lamberg2001numerical,
  title={Numerical solution of the Minkowski problem},
  author={Lamberg, L and Kaasalainen, M},
  journal={Journal of computational and applied mathematics},
  volume={137},
  number={2},
  pages={213--227},
  year={2001},
}

@book{evans2022partial,
  title={Partial differential equations},
  author={Evans, Lawrence C},
  volume={19},
  year={2022},
}

@book{evans2025measure,
  title={Measure theory and fine properties of functions},
  author={Evans, Lawrence C},
  year={2025},
}

@article{martinet2026,
  title={Numerical exploration of the range of shape functionals using neural networks},
  author={Martinet, Eloi and Ftouhi, Ilias},
  journal={arXiv preprint arXiv:2602.14881},
  year={2026}
}

@article{huang2025minkowski,
  title={Minkowski problems for geometric measures},
  author={Huang, Yong and Yang, Deane and Zhang, Gaoyong},
  journal={Bulletin of the American Mathematical Society},
  volume={62},
  number={3},
  pages={359--425},
  year={2025}
}

@article{burdzy2025,
  title={Geometric properties of optimizers for the maximum gradient of the torsion function},
  author={Burdzy, Krzysztof and Ftouhi, Ilias and Mariano, Phanuel},
  journal={arXiv preprint arXiv:2512.09400},
  year={2025}
}

@article{alves2021fundamental,
  title={Fundamental solutions for the Stokes equations: Numerical applications for 2D and 3D flows},
  author={Alves, Carlos JS and Serrao, Rodrigo G and Silvestre, Ana L},
  journal={Applied Numerical Mathematics},
  volume={170},
  pages={55--73},
  year={2021},
}

@article{marin2004method,
  title={The method of fundamental solutions for the Cauchy problem in two-dimensional linear elasticity},
  author={Marin, Liviu and Lesnic, Daniel},
  journal={International journal of solids and structures},
  volume={41},
  number={13},
  pages={3425--3438},
  year={2004},
  publisher={Elsevier}
}

@article{brasco2016spectral,
  title={Spectral inequalities in quantitative form},
  author={Brasco, Lorenzo and De Philippis, Guido},
  journal={arXiv preprint arXiv:1604.05072},
  year={2016}
}

@article{fusco2008sharp,
  title={The sharp quantitative isoperimetric inequality},
  author={Fusco, Nicola and Maggi, Francesco and Pratelli, Aldo},
  journal={Annals of mathematics},
  pages={941--980},
  year={2008},
  publisher={JSTOR}
}

@article{yu2018deep,
  title={The deep Ritz method: a deep learning-based numerical algorithm for solving variational problems},
  author={Yu, Bing and others},
  journal={Communications in Mathematics and Statistics},
  volume={6},
  number={1},
  pages={1--12},
  year={2018},
  publisher={Springer}
}

@article{boroczky2013volume,
  title={On the volume product of planar polar convex bodies—lower estimates with stability},
  author={B{\"o}r{\"o}czky, K and Makai, E and Meyer, Mathieu and Reisner, Shlomo},
  journal={Studia Scientiarum Mathematicarum Hungarica},
  volume={50},
  number={2},
  pages={159--198},
  year={2013},
  publisher={Akad{\'e}miai Kiad{\'o}}
}

@article{puny2021frame,
  title={Frame averaging for invariant and equivariant network design},
  author={Puny, Omri and Atzmon, Matan and Ben-Hamu, Heli and Misra, Ishan and Grover, Aditya and Smith, Edward J and Lipman, Yaron},
  journal={arXiv preprint arXiv:2110.03336},
  year={2021}
}

@article{grossmann2024can,
  title={Can physics-informed neural networks beat the finite element method?},
  author={Grossmann, Tamara G and Komorowska, Urszula Julia and Latz, Jonas and Sch{\"o}nlieb, Carola-Bibiane},
  journal={IMA Journal of Applied Mathematics},
  volume={89},
  number={1},
  pages={143--174},
  year={2024},
  publisher={Oxford University Press}
}

@article{mcgreivy2024weak,
  title={Weak baselines and reporting biases lead to overoptimism in machine learning for fluid-related partial differential equations},
  author={McGreivy, Nick and Hakim, Ammar},
  journal={Nature machine intelligence},
  volume={6},
  number={10},
  pages={1256--1269},
  year={2024},
  publisher={Nature Publishing Group UK London}
}

@article{geuzaine2009gmsh,
  title={Gmsh: A 3-D finite element mesh generator with built-in pre-and post-processing facilities},
  author={Geuzaine, Christophe and Remacle, Jean-Fran{\c{c}}ois},
  journal={International journal for numerical methods in engineering},
  volume={79},
  number={11},
  pages={1309--1331},
  year={2009},
  publisher={Wiley Online Library}
}

@article{gustafsson2020scikit,
  title={scikit-fem: A Python package for finite element assembly},
  author={Gustafsson, Tom and Mcbain, Geordie Drummond},
  journal={Journal of Open Source Software},
  volume={5},
  number={52},
  pages={2369},
  year={2020}
}

@article{barnett2008stability,
  title={Stability and convergence of the method of fundamental solutions for Helmholtz problems on analytic domains},
  author={Barnett, Alex H and Betcke, Timo},
  journal={Journal of Computational Physics},
  volume={227},
  number={14},
  pages={7003--7026},
  year={2008},
  publisher={Elsevier}
}

@article{dissanayake1994neural,
  title={Neural-network-based approximations for solving partial differential equations},
  author={Dissanayake, MWM Gamini and Phan-Thien, Nhan},
  journal={communications in Numerical Methods in Engineering},
  volume={10},
  number={3},
  pages={195--201},
  year={1994},
  publisher={Wiley Online Library}
}


\appendix

\section{Reminder of convex analysis}
\label{seq:convex_analysis}

Here we give some useful results of convex analysis. We begin with the definition of the convex conjugate of a function, that can be found in \cite{boyd2004convex}.

\begin{definition}[Convex conjugate]
    Let $f: \R^d \to (-\infty, +\infty]$. The convex conjugate of $f^* : \R^d \to (-\infty, +\infty]$ is defined as
    \begin{equation}
        \label{eq:conjugate}
        f^*(y) := \sup_{x \in \text{dom}(f)} \{ y \cdot x - f(x)\}
    \end{equation}
    where $\text{dom}(f)$ is the set of $x \in \R^d$ such that $f(x) < +\infty$.
\end{definition}

The following proposition can be found in \cite[Exercise 3.39]{boyd2004convex}:
\begin{proposition}
    \label{prop:bi-conjugate}
    Let $f: \R^d \to [-\infty, +\infty]$ be a proper convex and closed function. Then $f^{**} = f$.
\end{proposition}

In what follows, it will be useful to compute the composition of a convex function with a linear layer. The behavior of the convex conjugate with respect to the composition needs the following definition \cite[Definition 12.34]{Bauschke2017}:
\begin{definition}[Infimal postcomposition]
    Let $f: \R^d \to [-\infty, +\infty]$ and $L: \R^d \to \R^m$. The \textit{infimal postcomposition} of $f$ by $L$ is 
    \begin{align*}
      L \triangleright f\colon \R^m & \longrightarrow  [-\infty, +\infty] \\[-1ex]
      y & \longmapsto \inf_{Lx = y} f(x)
    \end{align*}
\end{definition}

In our case, we will only need the conjugate of the composition with a linear map, which is given by the following proposition (adapted from \cite[Proposition 13.24]{Bauschke2017}):
\begin{proposition}[Composition with a linear map]
    \label{prop:composition}
    Let $f: \R^d \to (-\infty, +\infty]$ and $A \in \R^{n\times m}$. Then
    \[
        (f \circ A)^* \leq A^T \triangleright f^*.
    \]
\end{proposition}

Finally, we will need to know the conjugate of the log-sum-exp function:
\begin{proposition}[Convex conjugate of the log-sum-exp]
    \label{prop:entropy}
    The convex conjugate of the log-sum-exp is the negative entropy function $-S$ where
    \begin{align*}
      S \colon \Delta^N & \longrightarrow (0, +\infty) \\[-1ex]
      y & \longmapsto -\sum_{i} y_i \log y_i
    \end{align*}
    where $\Delta^N$ is the $N$-dimensional simplex.
\end{proposition}

This can be shown by computing the optimality conditions in \eqref{eq:conjugate} (see \cite[Example 3.25]{boyd2004convex} for more details).

\section{Proofs}
\label{sec:proofs}

\begin{proof}[Proof of \cref{thm:uat}]
    We need to split the cases according to the definition of $\Om_\theta$. First, we define for $K,L \in \mathcal{K}$ the Hausdorff distance as $d_H(K,L) := \sup \left\{ \sup_{x \in K} d(x, L), \sup_{x \in L} d(x, K)\right\}$. 
    
    \textbf{Support function case:} Let $K \in \mathcal K$, $\eps > 0$. According to \cite[Theorem 1.8.19]{schneider}, there exist a polytope $P = \text{Conv}(w_1, \dots, w_m)$ such that $d_H(K, P) \leq \eps/2$. It is well known that the support function of $p$ is given by $h_P(x) := \sup_{ 1 \leq i \leq m} w_i \cdot x$. Now, let $p_\theta$ as defined in \cref{eq:lse_net} be such that $\theta = \{\beta, W_\beta\}$ where $W_\beta := \beta^{-1} ( w_1 \dots w_m )^T$ and let $\Om_\theta$ be defined by \cref{eq:support_nn}. This implies in particular that $h_{\Om_\theta} = p_\theta$. Since the log-sum-exp approximates the maximum function, we can take $\beta$ small enough so that $\|h_P - p_\theta\|_{C(\S^{n-1})} \leq \eps/2$. Moreover, according to \cite[Lemma 1.8.14]{schneider}, $d_H(P,\Om_\theta) = \|h_P - p_\theta\|_{C(\S^{n-1})}$. Putting everything together, we get that there exists $\theta$ such that
    \[
        d_H(K, \Om_\theta) \leq d_H(K, P) + d_H(P, \Om_\theta) \leq \eps
    \]
    hence $\mathcal{K}^\text{NN}$ is dense in $\mathcal{K}$.
    
    \textbf{Gauge function case:}
    We first prove the following fact: for $K,L \in \mathcal{K}$, we have $d_H(K,L) \leq \|\rho_K - \rho_L\|_{C(\S^{n-1})}$ where
    \begin{align*}
      \rho_K \colon \S^{n-1} & \longrightarrow  \R^+\\[-1ex]
      x & \longmapsto \sup\{r > 0 : rx \in K\} 
      \end{align*}
    is the \textit{radial function} of $K$.
    In particular, it is immediate that the radial function is the inverse of the gauge function on $\S^{n-1}$.
    For $x \in K$, let us define 
    \[
        y = \begin{cases}
            \rho_L\left(\frac{x}{\|x\|}\right)\frac{x}{\|x\|} &\mbox{ if } x \not\in L\\
            x &\mbox{ if } x \in L
        \end{cases}.
    \]
    Then $y \in L$ and we have
    \begin{align*}
        d(x, L) 
        &\leq \max \left\{0, \|x\| - \rho_L\left(\frac{x}{\|x\|}\right)\right\} 
        \leq \left|\rho_K\left(\frac{x}{\|x\|}\right) - \rho_L\left(\frac{x}{\|x\|}\right)\right\|
        \leq \|\rho_K - \rho_L\|_{C(\S^{n-1})}
    \end{align*}
    Hence $\sup_{x\in K}d(x, L) \leq \|\rho_K - \rho_L\|_{C(\S^{n-1})}$ . By symmetry, we deduce the claim.

    Now, let $K \in \mathcal K$, $\eps > 0$. Let $P$ be a polytope such that $P \subset K$ and $d_H(K, P) \leq \eps$. According to \cite[Theorem 2.4.3]{schneider}, there exists $m>0$ and vectors $w_1, \dots, w_m \in \R^d$ such that
    \[
        P := \bigcap_{1\leq i\leq m} \left\{x \in \R^d : w_i \cdot x \leq 1 \right\}.
    \]
    In particular, we have $g_P(x) = \max_{1 \leq i \leq m} w_i \cdot x$. The condition $P \subset K$ implies that $0 < \alpha := \min_{u \in \S^{n-1}} g_K(u) \leq \min_{u \in \S^{n-1}} g_P(u)$. Similarly as before, we can take $\beta$ small enough so that 
    \[
        \|g_P - p_\theta\|_{C(\S^{n-1})} \leq \eps \quad \mbox{and} \quad \min_{u \in \S^{n-1}} p_\theta(u) \geq \alpha/2.
    \]
    Hence, for $\Om_\theta$ defined as in \cref{eq:gauge_nn}, we have
    \begin{align*}
        \|\rho_{\Om_\theta} - \rho_P\|_{C(\S^{n-1})}
        =  \left\|\frac{1}{p_\theta} - \frac{1}{g_P}\right\|_{C(\S^{n-1})} 
        \leq  \frac{\eps}{\min_{\S^{n-1}} p_\theta \min_{\S^{n-1}} g_P}
        \leq \frac{2}{\alpha^2} \eps\\ 
    \end{align*}
    leading to $d_H(K, \Om_\theta) \leq d_H(K, P) + d_H(P, \Om_\theta) \leq \left(1 + \frac{2}{\alpha^2}\right) \eps$.
\end{proof}

\begin{proof}[Proof of proposition \labelcref{prop:symmetries}]
    We will use \cref{eq:convex_from_func}. Let $p_\theta$ be a sublinear network.
    
    \textbf{Case 1 (gauge):} assume that $\Om_\theta := \left\{ x \in \R^d : p_\theta^G(x) \leq 1 \right\}$. For $x \in \Om_\theta$ and $g \in G$, we have that $p_\theta^G(g.x) = \frac{1}{|G|} \sum_{\tilde g \in G} p_\theta(g.(\tilde g.x)) = \frac{1}{|G|} \sum_{\tilde g \in G} p_\theta((g\tilde g).x)  = p_\theta^G(x) \leq 1$ hence $g.\Om_\theta \subset \Om_\theta$. We can then deduce the reverse inclusion by $x \in \Om_\theta \implies g^{-1}.x \in \Om_\theta \implies x \in g.\Om_\theta$.

    \textbf{Case 2 (support):} assume that $\Om_\theta := \left\{ x \in \R^d : x \cdot y \leq p_\theta^G(y) \text{ for all } y \in \R^d \right\}$. For $x \in \Om_\theta$ and $g \in G$, we have for all $y \in \R^d$:
    \[
        (g.x) \cdot y = x \cdot (g^{-1}.y) = \leq p_\theta^G(g^{-1}.y) = p_\theta^G(y)
    \]
    hence $g.\Om_\theta \subset \Om_\theta$. The reverse inclusion follows.
\end{proof}

\section{Computation of shape quantities}
\subsection{Geometric--differential quantities}
\label{sec:geo_diff}

\paragraph{Expression of the normal vector} As it has been stated, the normal vector is defined by 
\[
    n_\theta\left(y \right) = \frac{\left(D\phi_\theta\right)^{-T}(x) n_B(x)}{\left\|\left(D\phi_\theta\right)^{-T}(x) n_B(x)\right\|}.
\]
for $y = \phi_\theta(x)$, $x \in \partial B$. Indeed, if $\varphi$ is a level set function associated to a smooth set $\Om$ (i.e. $\Om = \{ \varphi \leq 0 \}$), then for $y \in \dOm$, the normal vector can be computed as
\[
    n(y) = \frac{\nabla \varphi(y)}{\|\nabla \varphi(y)\|}.
\]
In the case of $\Om_\theta$, the function $\varphi(y) = \|\phi_\theta^{-1}(y)\|^2 - 1$ is an admissible level set function since $\varphi(y) \leq 0 \iff y \in \Om_\theta$. Computing the gradient leads to $\nabla \varphi(y)  = D \left(\phi_\theta^{-1}\right)^T(y).2\phi_\theta^{-1}(y) = 2 \left(D\phi_\theta\right)^{-T}(x) x$, from which we deduce the formula by normalization. Notice that we never need the expression of $\phi_\theta^{-1}$ to compute $n_\theta(y)$ when $y=\phi_\theta(x)$.

\paragraph{Numerical computation of the Weingarten map} The Weingarten map is defined as
\begin{align*}
    S_y \colon T_y \partial \Om_\theta &\longrightarrow  T_y \partial \Om_\theta\\
    v &\longmapsto D_{\Gamma} n(y).v
\end{align*} where $D_{\Gamma} n(y)$ is the tangential derivative of $n$, i.e.  the restriction of $D n(y) \in \R^{d \times d}$ to the tangent space $T_y \partial \Om_\theta$. In order to compute it numerically, one introduces the Housholder matrix $H_y = I - vv^T$ where 
\[ 
    v = \frac{n(y)-e_d}{\|n(y)-e_d\|},
\]
$e_1, \dots, e_d$ being the canonical basis of $\R^d$. Hence, $H_yn=e_d$ and in particular, the first $d-1$ lines of $H_y$ are orthogonal to $n(y)$ and hence forms an orthonormal basis of $T_y \partial \Om_\theta$. One can then compute $D_{\partial \Om_\theta} n(y)$ in this basis as 
\[
    D_{\partial \Om_\theta} n(y) = \tilde H_y Dn(y) \tilde H_y^T \in \R^{(d-1) \times (d-1)}
\]
where $\tilde H = (H_{ij})_{\substack{1 \leq i \leq d-1\\1 \leq j \leq d}}$.

\subsection{Method of fundamental solutions}
\label{seq:mfs_galerkin}

An important special case of \cref{eq:poisson} is the case where $f$ is identically equal to $1$, i.e.
\begin{equation}
    \label{eq:torsion}
    \begin{cases}
        -\Delta u &=\ \  1\ \ \ \mbox{ in } \Om_\theta,\\
        \ \ \ \ \ u &=\ \  0\ \ \   \mbox{ on } \dOm_\theta.
    \end{cases}
\end{equation}
In this case, the solution $u\in H^1_0(\Om_\theta)$ is called the \textit{torsion function}, which is an important function in mechanics that has been extensively studied in mathematics. One important derived quantity is the \textit{torsional rigidity} defined as 
\begin{equation}
    \label{eq:torsional_rigidity}
    T(\Om_\theta) = \int_{\Om_\theta} u
\end{equation}
which describes the rigidity of a rod of cross section $\Om_\theta$.

It is well known that \cref{eq:torsion} can be solved in an extremely precise and efficient way using the \textit{method of fundamental solutions} \cite{barnett2008stability, Antunes2022May, Bogosel2016Nov}. Indeed, by putting $\phi(x) = u(x) + \frac{x_1^2}{2}$, \cref{eq:torsion} can be equivalently reformulated as the following Laplace equation:
\begin{equation}
    \begin{cases}
        -\Delta \phi &= \ \ 0\ \ \ \  \mbox{ in } \Om_\theta,\\
        \ \ \ \ \ \phi &= \ \ \frac{x_1^2}{2}\ \ \    \mbox{ on } \dOm_\theta.
    \end{cases}
\end{equation}
One may then seek an approximate solution $\tilde \phi$ expressed as a linear combination of fundamental solutions, namely,
\[
    \tilde \phi(x) := \sum_{i=1}^n c_i \psi(x - y_i),
\]
where $\psi$ is the fundamental solution to $-\Delta \psi = \delta_0$ in $\R^d - \{0\}$ and $y_1,\dots, y_n \in \Om_\theta^c$ \cite{evans2022partial}. Since $\phi$ is harmonic in $\Om$, we only have to fit the boundary condition, for instance, in an $L^2$ sense, which amounts at solving
\[
    \min_{c_1, \dots, c_n \in \R} \int_{\dOm} \left| \tilde \phi(x) - \frac{x_1^2}{2} \right|^2d\sigma.
\]
Since this integral cannot be analytically computed for a general $\Om$, it is discretized and the resulting least squares problem is solved using \verb|torch.linalg.lstsq|. 

There exists several valid choices for the placement of the sources $y_1, \dots, y_n$. In our case, we decided to do it the following way: draw samples $x_1, \dots, x_n \in \partial B$, and set $y_i := \phi_\theta(x_i) + \eps n_\theta(x_i)$ for $\eps > 0$. This way, we have the guarantee that $y_i$ lies outside and at distance $\eps$ of $\Om_\theta$. An alternative and more flexible approach in the case where $\Om_\theta$ is parametrized by a gauge function \cref{eq:gauge_nn} is to take $y_i = \phi_\theta\left((1+\eps)x_i\right)$. This has the advantage that scaling $\Om_\theta$ leads to the same scaling of the distance from $y_i$ to $\Om_\theta$.

It is well known that this method is very sensitive to the source placement, and that the conditioning worsens when the number of sources increase. However, this can be mitigated by an \textit{a posteriori} error estimation using the maximum principle: knowing that $\|\phi-\tilde\phi\|_\infty := \sup_{\Om_\theta} |\phi - \tilde \phi| = \sup_{\dOm_\theta} |\phi - \tilde \phi| = \sup_{x \in \dOm_\theta} \left|\frac{x_1^2}{2} - \tilde \phi\right|$, we can estimate the last quantity by sampling a large amount of points on $\dOm_\theta$. If the error is greater than a certain tolerance (taken in our experiments of the order $10^{-4}-10^{-5}$), we can modify the source placement parameter $\eps$ and increase the number of sources until the tolerance is reached.

\begin{remark}
    While we decided to treat the case of the Laplace equation for simplicity, this approach can easily be extended to other, more practical settings where we have access to the fundamental solutions of the PDE, like in linearized elasticity \cite{marin2004method} for the optimal design of structures or in the Stokes problem \cite{alves2021fundamental}.
\end{remark}


\section{Comparison with PINNs}
\label{seq:pinn}

As it has already been stressed, PINNs are not particularly well suited in general for the computation of PDE-constrained shape optimization. However, there is some very specific problems for which they can easily be applied, namely when the objective function is the \textit{Dirichlet energy} of the shape. This is the type of energy that has been considered in \cite{BelieresFrendo2025Apr}. In this section, we will compare the use of PINNs with the other methods previously introduced on a similar problem for which we know the analytical solution.

Consider the following problem:
\begin{equation}
    \label{eq:saint_venant}
    \max_{\substack{\Om \in \mathcal K \\ \Vol(\Om) = 1}} T(\Om)
\end{equation}
where $T$ is the torsional rigidity defined in \cref{eq:torsional_rigidity}. A foundational result in shape optimization, called the Saint--Venant inequality, states that the solution of this problem is the ball \cite{brasco2016spectral}.

The Dirichlet energy associated with \cref{eq:torsion} is defined as
\begin{equation}
    \label{eq:dirichlet_energy}
    E(\Om, v) := \frac{1}{2} \int_\Om |\nabla v|^2 - \int_\Om v
\end{equation}
for $v \in H^1_0(\Om)$. A fundamental property of the Dirichlet energy is that the minimizer $u_\Om$ of $E(\Om, \dot)$ is the solution of \cref{eq:torsion}. Using the fact that $\int_\Om |\nabla u_\Om|^2 = \int_\Om u_\Om = T(\Om)$, this implies that $E(\Om, u_\Om) = -\frac{T(\Om)}{2}$. Hence:
\[
    \min_{\substack{\Om \in \mathcal K, v \in H^1_0(\Om) \\ \Vol(\Om) = 1}} E(\Om, v) = \min_{\substack{\Om \in \mathcal K \\ \Vol(\Om) = 1}} \min_{v \in H^1_0(\Om)} E(\Om, v) = \min_{\substack{\Om \in \mathcal K \\ \Vol(\Om) = 1}} E(\Om, u_\Om) = - \frac{1}{2} \max_{\substack{\Om \in \mathcal K \\ \Vol(\Om) = 1}} T(\Om), 
\]
which means that minimizing $E$ in both $\Om$ and $v$ is actually equivalent to \cref{eq:saint_venant}. Using the following proposition, we can further reformulate the problem:
\begin{proposition}
    Any minimizer $(\Om, v)$ of 
    \begin{equation}
        \label{eq:dirichlet_constrained}
        \min_{\substack{\Om \in \mathcal K, v \in H^1_0(\Om) \\ \Vol(\Om) = 1}} E(\Om, v)
    \end{equation}
    is a minimizer of 
    \begin{equation}
        \label{eq:dirichlet_rescaled}
        \min_{\Om \in \mathcal K, v \in H^1_0(\Om)} \frac{E(\Om, v)}{\Vol(\Om)^\frac{d+2}{d}}.
    \end{equation}
    Reciprocally, if $(\Om, v)$ is a minimizer of \cref{eq:dirichlet_rescaled} then $\left(\alpha_\Om \Om, \alpha_\Om^2 v(\cdot / \alpha_\Om)\right)$ where $\alpha_\Om = \Vol(\Om)^{-1/d}$ is a minimizer of \cref{eq:dirichlet_constrained}.
\end{proposition}

\begin{proof}
    Using the change of variable formula, we can show that for $\alpha>0$,
    \[
        E\left(\alpha \Om, \alpha^2 v(\cdot/\alpha) \right) = \alpha^{d+2} E(\Om, v).
    \]
    The proposition follows.
\end{proof}

Therefore, we can actually drop the volume constraint and instead solve \cref{eq:dirichlet_rescaled}. For a parametrized set $\Om_\theta$ and $v \in H^1_0(\Om_\theta)$, we have 
\[ 
    E(\Om_\theta, v) = \frac{1}{2} \int_B A_\theta \nabla (v \circ \phi_\theta^{-1}) \cdot \nabla (v \circ \phi_\theta^{-1}) - \int_B (\Jac \phi_\theta) (v \circ \phi^{-1});
\]
moreover, minimizing this last expression with respect to $v \in H^1_0(\Om_\theta)$ is equivalent to minimizing 
\[ 
    F(\Om_\theta, v) = \frac{1}{2} \int_B A_\theta \nabla v \cdot \nabla v - \int_B (\Jac \phi_\theta) v.
\]
for $v \in H^1_0(B)$. Finally, by taking $v$ to be a neural network $v_\eta$ where $\eta$ are its parameters, we can approximate \cref{eq:saint_venant} by equivalently minimizing the loss 
\[  
    L(\theta, \eta) := \frac{F(\Om_\theta, v_\eta)}{\Vol(\Om)^\frac{d+2}{d}}.
\]

\begin{remark}
    In order to exactly impose the Dirichlet boundary condition on $\partial B$, we chose $v_\eta$ of the form $v_\eta(x) = \text{dist}_B(x) \text{MLP}_\eta(x)$ where $\text{dist}_B$ is the signed distance function of $B$.
\end{remark}

This approach follows the line of the so-called \textit{Deep Ritz Method} \cite{yu2018deep}. In opposition to what was previously described in this paper and to classical shape optimization algorithms, one does not need to fully solve the state equation before updating the shape; the state and the shape are jointly optimized, and one can expect $v_\theta$ to be close to the solution of the state equation only at convergence. Since we do not have to solve a PDE at each iteration, one could expect a certain speedup of this method compared to classical ones; it is however not the case, as is shown in the next experiment.

\paragraph{Evaluation of the method:} We compare the speed and accuracy of the previously described method with the maximization of $\frac{T(\Om_\theta)}{\Vol(\Om_\theta)^\frac{d+2}{d}}$ where $T(\Om_\theta)$ is computed using either the method of fundamental solutions or the mesh free Galerkin method. While the former is expected to be extremely fast and accurate, it doesn't generalize well to other problems with non-constant source term. On the contrary, the latter is expected to be slower due to the fact that it is not taylor to this particular problem; the tradeoff being that it is able to treat a large variety of problems. 

Since we know that the optimal shape is the ball, we can measure the discrepency between the shape $\Om_\theta$ and $B$ using the \textit{isoperimetric deficit} \cite{fusco2008sharp}, defined as 
\[
    D(\Om_\theta) = c_d \frac{\Per(\Om_\theta)}{\Vol(\Om_\theta)^\frac{d-1}{d}} - 1
\]
with $c_d = (d \Vol(B)^{1/d})^{-1}$. The smallest the isoperimetric deficit, the closest $\Om_\theta$ is from the ball. 

In order to give the best advantage to the PINN approach, we first perform a hyperparameter search for the PINN-based method. The tested parameters are given in \cref{tab:hyper}, along with the two best configurations, that we picked in our final experiment. The number of points used for integral evaluation is the same across all experiments ($n = 20000$). The evaluation of the perimeter and volume for the computation of the isoperimetric deficit is made using $n=100000$ points. On the $72$ possible choices of parameters, only $3$ reached the precision of $10^{-5}$ eventually; we decided to keep the three of them to compare with the classical methods. These configurations are decribed in \cref{tab:hyper}.

\begin{table}[h]
    \centering
    \caption{Hyperparameter search space (left) and best configurations that are kept for the final experiment.}
    \begin{tabular}{l c c c c }
        \toprule
        Hyperparameter & Values & Config 1 & Config 2 & Config 3 \\
        \midrule
        Neurons per layer & $\{32, 64\}$ & $32$ & $32$ & $32$\\
        Depth & $\{2,3,4\}$ & $3$ & $3$ & $2$\\
        Activation Function & \{ sin, tanh \} & tanh & tanh & tanh\\
        Optimizer & \{Adam, L--BFGS\} & L--BFGS & Adam & Adam\\
        Learning rate & $\{10^{-1}, 10^{-2}, 10^{-3} \}$ & $10^{-2}$ & $10^{-2}$ & $10^{-2}$ \\
        \bottomrule
    \end{tabular}
    \label{tab:hyper}
\end{table}

In \cref{fig:comparison}, we show the comparison between the two classical methods and the three PINN-based ones, both on CPU (AMD Ryzen Pro 7) and on GPU (Nvidia L40). We see that the classical methods consistently outperforms the PINN-based one; interestingly enough, both classical methods ran on CPU are still faster and more accurate than the PINN-based methods ran on GPU. This experiment show that, even in the most favorable setting, the PINN approach is outperformed by the classical methods by several orders of magnitude.

\begin{figure}
    \centering
    \includegraphics[width=0.48\linewidth]{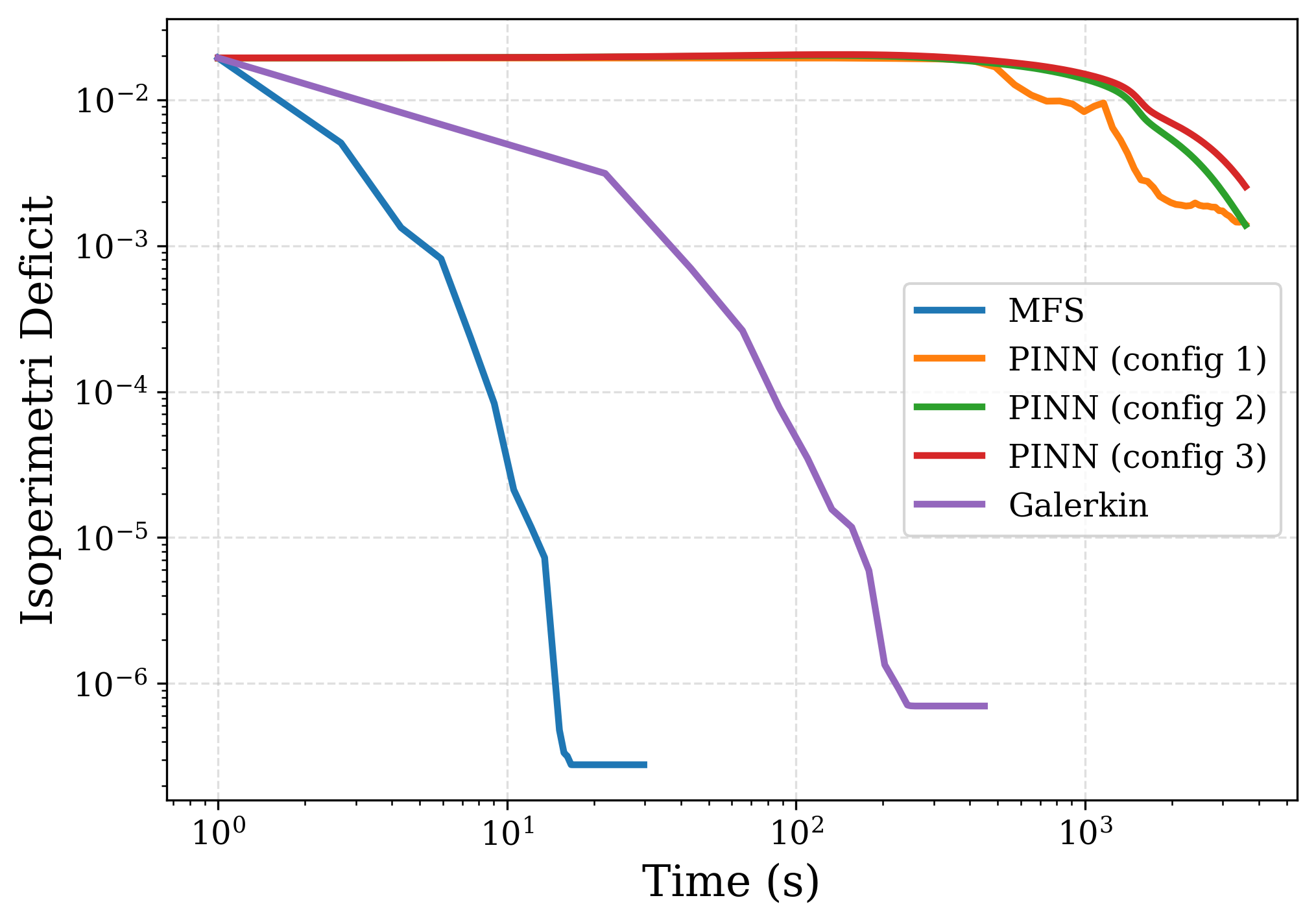}
    \hfill
    \includegraphics[width=0.48\linewidth]{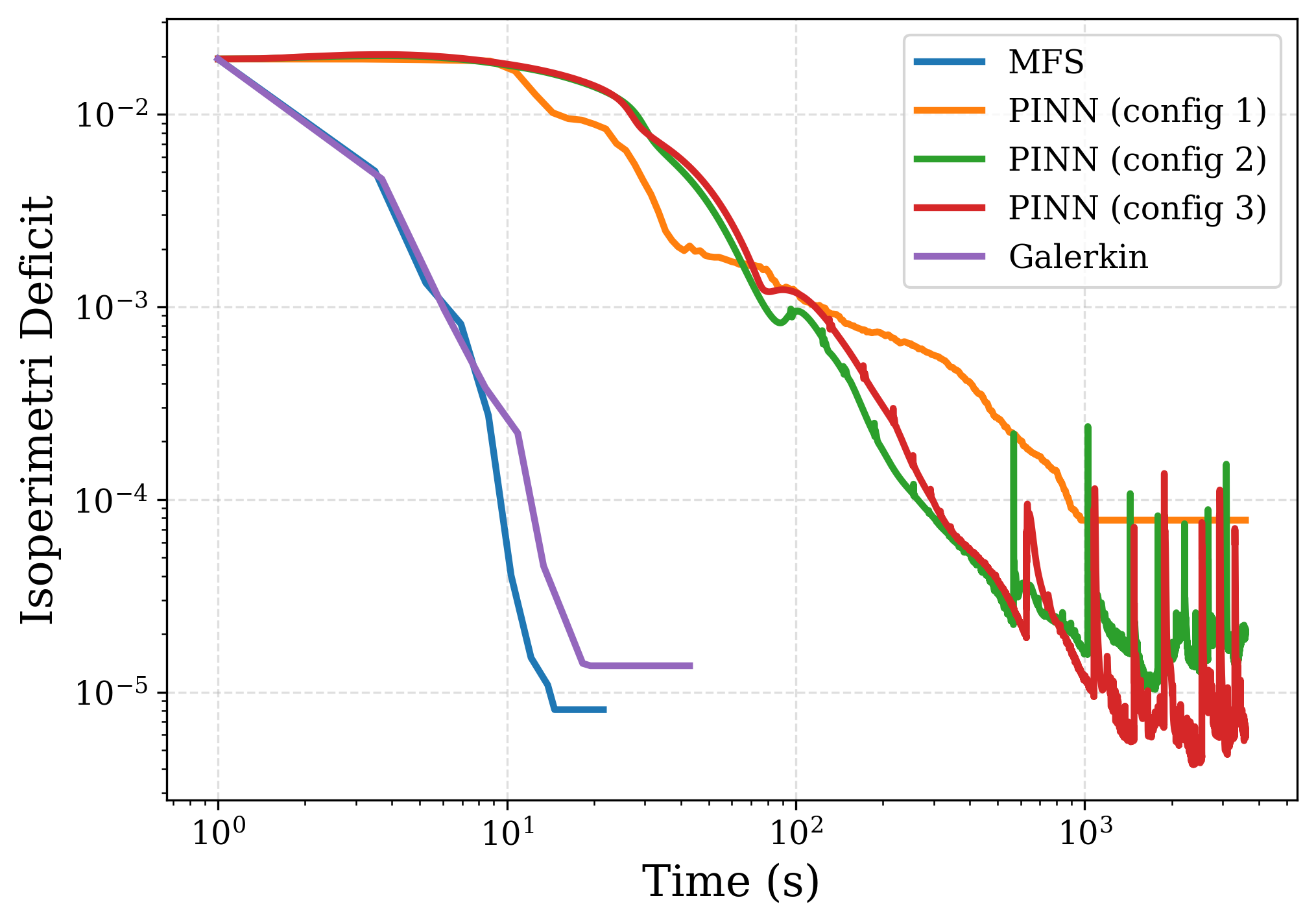}
    \caption{Runtime vs accuracy comparison between classical solvers and PINN-based approaches (CPU/GPU). Classical methods achieve higher accuracy at significantly lower computational cost.}
    \label{fig:comparison}
\end{figure}

\section{Minimization of the Mahler volume}
\label{sec:mahler}

Gauge and support functions are related to each other through the notion of \textit{polar body}. The polar body of a convex set $K \in \mathcal{K}$ is defined as $K^\circ := \left\{ x \in \R^d : x \cdot y \leq 1 \text{ for all } y \in K \right\}$. Then, the gauge function of $K$ is the support function of $K^\circ$ (see \cite[Lemma 1.7.13]{schneider}), i.e. $g_{K} = h_{K^\circ}$.

In this section, we are interested in the minimization of the \textit{Mahler volume} which is the product of the volumes of a convex $\Om$ and its polar body $\Om^\circ$, i.e. $\Vol_M(\Om) := \Vol(\Om) \Vol(\Om^\circ)$. Our framework allows us to easily compute both volumes using simultaneously the support and the gauge parametrization.

In $\R^d$, it is known that the convex set with $n$-fold rotational symmetry that minimizes the Mahler volume is the regular polygon, and the optimal value is $n^2 \sin(\pi/n)^2$ \cite{boroczky2013volume}. Using our parametrization of symmetric convex sets, we get the results presented in \cref{fig:mahler_2d}. As you can see, our parametrization is flexible enough to recover polygons.

\begin{figure}
    \centering
    \includegraphics[width=0.19\linewidth]{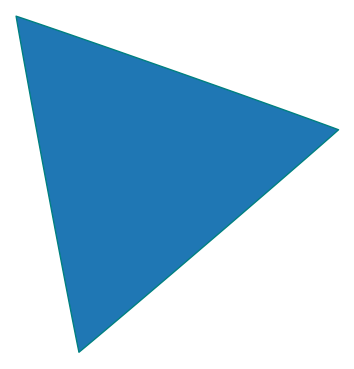}
    \includegraphics[width=0.19\linewidth]{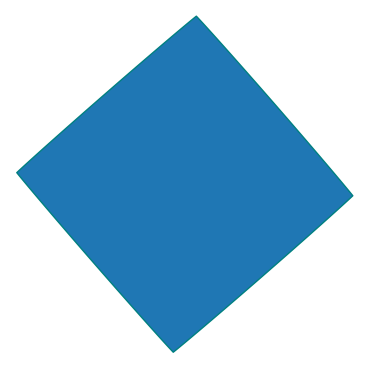}
    \includegraphics[width=0.19\linewidth]{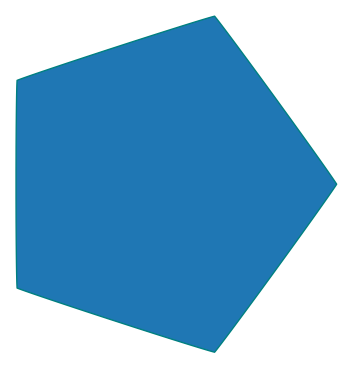}
    \includegraphics[width=0.19\linewidth]{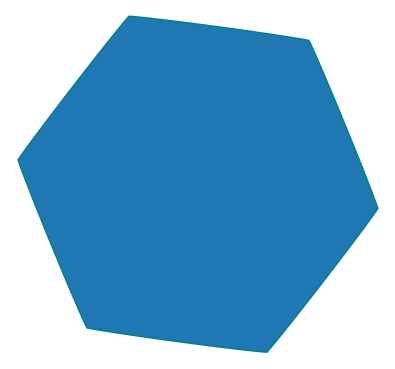}
    \includegraphics[width=0.19\linewidth]{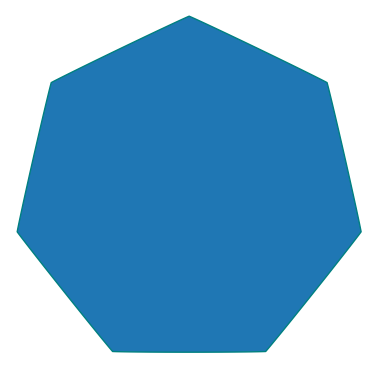}
    \caption{Minimizers of the Mahler volume under $n$-fold symmetry constraints. The method recovers polygonal structures, highlighting its ability to represent non-smooth convex shapes.}
    \label{fig:mahler_2d}
\end{figure}

\section{Statistical Analysis}
\label{seq:fit_statistic}

We complement the experiments of \cref{subseq:fit} with a statistical evaluation of robustness to noise. For each configuration (shape and noise level), we perform the following procedure. We randomly initialize the sublinear network $p_\theta$, sample $n=1000$ points ${y_i}_{i=1}^n$, and optimize the loss defined in \cref{subseq:fit}.

To quantify reconstruction quality, we measure the $L^2$ discrepancy between the learned function $p_\theta$ and the gauge function of the target shape $g_{\text{target}}$. Since $g_{\text{target}} = 1$ on $\partial \Omega_{\text{target}}$, this reduces to
\[
    \text{Acc}(\theta) =  \|g_\text{target} - p_\theta \|_{L^2(\partial \Om_\text{target})} = \|p_\theta - 1 \|_{L^2(\partial \Om_\text{target})} 
\]
We approximate this quantity by sampling $10^5$ points uniformly on $\partial B$ and mapping them via $\phi_{\text{target}}$.

Each experiment is repeated $100$ times with independent random initializations and samples. We report the median performance together with the interquartile range (25th–75th percentiles). The aggregated results are shown in \cref{fig:statistic}.

\begin{figure}
    \centering
    \includegraphics[width=0.7\linewidth]{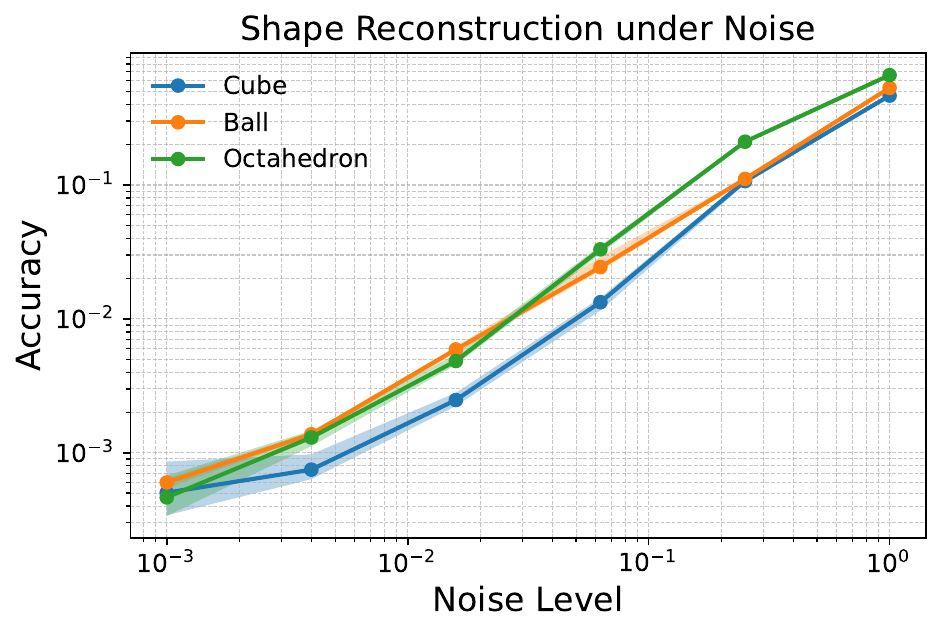}
    \caption{Robustness to noise and network initialization across 100 runs. Median reconstruction error with interquartile range. The method exhibits low variance and stable performance across noise levels.}
    \label{fig:statistic}
\end{figure}

In a second experiment, we analyze the influence of the number of samples on the reconstruction. We fix the noise level to $\sigma = 0.01$ and use different amount of points $n$ ranging from $10$ to $10^4$. The results are reported in \cref{fig:statistic_number_of_samples}.
\begin{figure}
    \centering
    \includegraphics[width=0.7\linewidth]{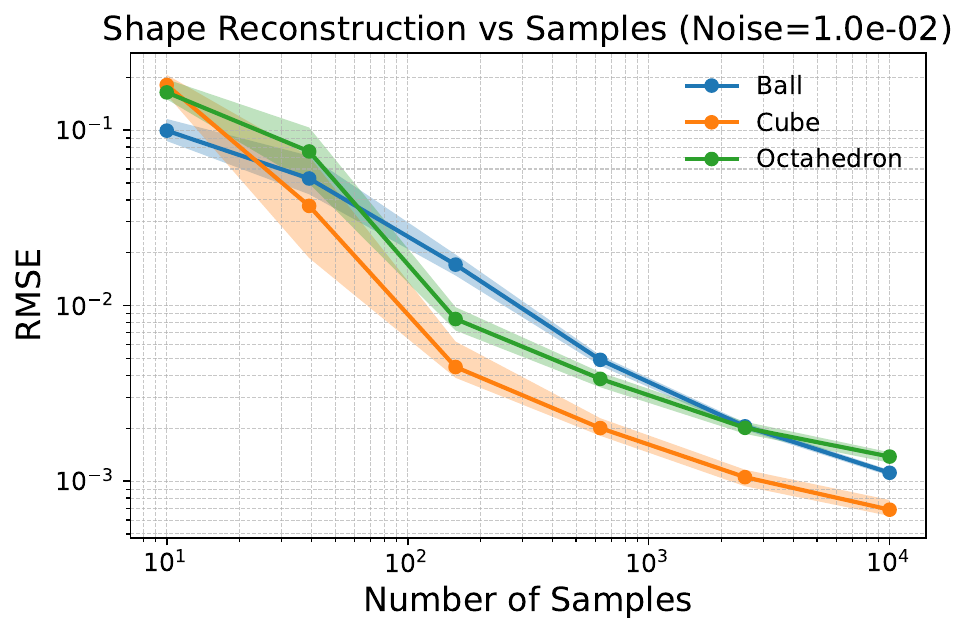}
    \caption{Robustness to amounts of samples and network initialization across 100 runs. Median reconstruction error with interquartile ranges. As expected, the variance is higher when using fewer samples.}
    \label{fig:statistic_number_of_samples}
\end{figure}

\section{Integration of the finite element method}
\label{seq:fem}

Assume that we want to minimize a certain shape function $J$, depending on a PDE that we would like to solve using the finite element method. While possible, it requires to compute certain derivatives manually.

Defining $j(\theta) := J(\Om_\theta)$, we want to relate the derivatives $\partial_{\theta_k} j(\theta)$ with the shape derivative $dJ(\Om_\theta).V$. According to \cite[Chapter 5, Section 5.9]{HPb}, under some mild regularity assumptions, shape derivatives can be put in the form 
\begin{equation}
    \label{eq:hadamard}
    dJ(\Om).V = \int_\dOm f(x) V(x)\cdot n_\Om(x)
\end{equation}
where $f : \dOm \to \R$ depends on $\Om$. This formula is particularly well suited when one needs to compute $f$ \textit{via} a mesh-based. Now, let $\theta = (\theta_1, \dots, \theta_n) \in \R^d$ and define $\bar \theta(t) := (\theta_1, \dots, \theta_k +t, \dots, \theta_n)$. Following \cite[Chapter 5, Section 5.2]{HPb}, define 
\[
    \Phi(t)(x) = \phi_{\bar\theta(t)}\circ \phi_\theta^{-1}(x)
\]
for $t\in \R$ and $x \in \Om_\theta$. We have $\Phi(t)(\Om_\theta) = \phi_{\bar\theta(t)}\circ \phi_\theta^{-1}(\Om_\theta) = \phi_{\bar\theta(t)}\circ \phi_\theta^{-1}\circ \phi_\theta(B) = \Om_{\bar\theta(t)}$. Hence, by letting $V := \Phi'(0)$, we have by definition of the shape derivative that
\[
    dJ(\Om_\theta).V = \lim_{t \to 0} \frac{J(\Om_{\bar\theta(t)}) - J(\Om_\theta)}{t} = \lim_{t \to 0} \frac{j(\bar\theta(t)) - j(\theta)}{t} = \partial_{\theta_k}j(\theta).
\]
On the other hand, 
\[
    V(x) = \Phi'(0)(x) = \left. \frac{d \phi_{\bar\theta(t)}\circ \phi_\theta^{-1}(x)}{dt}\right|_{t=0} = \partial_{\theta_k} \phi_\theta \left( \phi_\theta^{-1}(x) \right),
\]
meaning that formally, we have
\begin{equation}
    \label{eq:parametric_derivative}
    \partial_{\theta_k}j(\theta) = \int_{\dOm_\theta} f(x) \partial_{\theta_k} \phi_\theta \left( \phi_\theta^{-1}(x) \right) \cdot n_{\Om_\theta}(x) dx.
\end{equation}
If $\phi^{-1}_\theta$ is easily computable (which is the case, for instance, for \cref{eq:gauge_nn} of inverse $\phi^{-1}_\theta(y) = \frac{p_\theta(y)}{\|y\|} y$) and one can precisely evaluate the boundary integral (for instance, if we have a mesh) then this integral can be easily computed. Otherwise, one can formulate it on the reference domain:
\begin{align*}
    \partial_{\theta_k}j(\theta)
    &= \int_{\partial B} f(\phi_\theta(x)) \partial_{\theta_k} \phi_\theta \left( \phi_\theta^{-1}(\phi_\theta(x)) \right) \cdot n_{\Om_\theta}(\phi_\theta(x)) \Jac_{\partial B}(\phi_\theta(x)) dx\\
    &= \int_{\partial B} \left(f\circ \phi_\theta\right) \partial_{\theta_k} \phi_\theta \cdot \frac{\left(D\phi_\theta\right)^{-T} n_B}{\left|\left(D\phi_\theta\right)^{-T} n_B\right|}\Jac_{\partial B}(\phi_\theta)\\
    &= \int_{\partial B} \left(f\circ \phi_\theta\right) \partial_{\theta_k} \phi_\theta \cdot \left(D\phi_\theta\right)^{-T} n_B\Jac(\phi_\theta)
\end{align*}

An example of optimal sets using this method can be found in \cref{fig:convex_dirichlet}, where we minimize the first $6$ Dirichlet eigenvalues under volume and convexity constraint. Specifically, we solve:
\begin{equation}
    \min_{\Vol(\Om)=1} \lambda_k(\Om) 
\end{equation}
where $\Om \in \mathcal{K}$ and $\lambda_k(\Om)$ is the $k$-th Dirichlet eigenvalue, i.e. it solves
\begin{equation}
    \begin{cases}
        -\Delta u = \lambda_k(\Om) u &\mbox{ in } \Om\\
        u = 0   &\mbox{ on } \dOm
    \end{cases}
\end{equation}
for some $u \in H^1_0(\Om)$. This particular problem has already been considered in \cite{Antunes2022May}.

In order to compute the derivatives \cref{eq:parametric_derivative} of $\lambda_k$, we created a mesh of the domain $\Om_\theta$ using \verb|gmsh| \cite{geuzaine2009gmsh}, by giving it a sequence of points $y_i = \phi_\theta(x_i)$ where $x_i = (\cos(2i\pi/n), \sin(2i\pi/n))^T$, $1 \leq i < n$. An eigenpair $\lambda(\Om_\theta),u(\Om_\theta)$ is computed \textit{via} finite elements using \verb|scikit-fem| \cite{gustafsson2020scikit}. Using the Hadamard expression for the shape derivative $d\lambda(\Om).V = \int_\dOm |\nabla u|^2 (V \cdot n)$, we deduce that
\[
    \partial_{\theta_k}j(\theta) = \int_{\dOm_\theta} |\nabla u(x)|^2 \partial_{\theta_k} \phi_\theta \left( \phi_\theta^{-1}(x) \right) \cdot n_{\Om_\theta}(x) dx.
\]
The differential quantity $\nabla u(x)$ is computed on the finite element space, while $\partial_{\theta_k} \phi_\theta$ is computed using AD. The integral is computed on the boundary mesh by \verb|scikit-fem|. All of this process is wrapped in  \verb|torch.autograd.Function| in order to seamlessly integrate it into PyTorch's AD.

 You can see the six optimal shapes found by the algorithm in \cref{fig:convex_dirichlet} along with the corresponding eigenvalue and the relative error $E_\text{rel} := \frac{\lambda_k - \lambda_k^*}{\lambda_k^*}$ where $\lambda_k^*$ is either the optimal value obtained numerically in \cite{Antunes2022May} or the analytical one for $k=1$ and $k=3$ where the optimal shape is known to be the ball. We see that the shapes are in good agreement with the previous results, and that the relative error is of order $-3$ to $-4$.

\begin{figure}[h!]
    \centering
    \begin{subfigure}{0.26\textwidth}
        \centering
        \includegraphics[width=\linewidth]{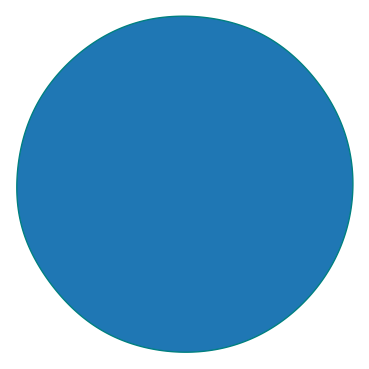}
        \caption*{$\lambda_1 = 18.1686$\\$E_{\text{rel}} = 7.7e-4$}
    \end{subfigure}
    \begin{subfigure}{0.26\textwidth}
        \centering
        \includegraphics[width=\linewidth]{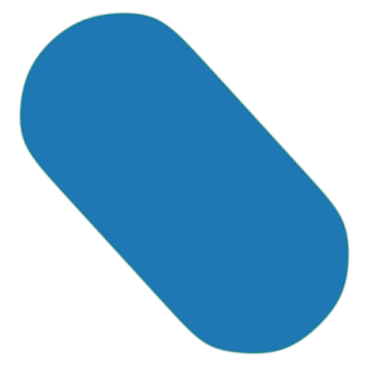}
        \caption*{$\lambda_2 = 38.0047$\\$E_{\text{rel}} = 1.2e-4$}
    \end{subfigure}
    \begin{subfigure}{0.26\textwidth}
        \centering
        \includegraphics[width=\linewidth]{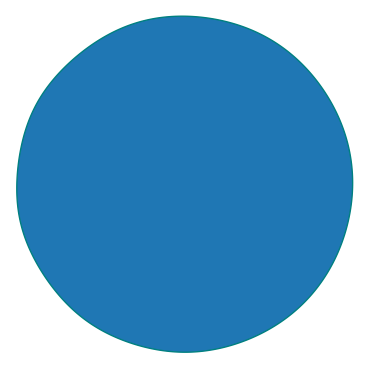}
        \caption*{$\lambda_3 = 46.1288$\\$E_{\text{rel}} = 2.3e-4$}
    \end{subfigure}
    \begin{subfigure}{0.26\textwidth}
        \centering
        \includegraphics[width=\linewidth]{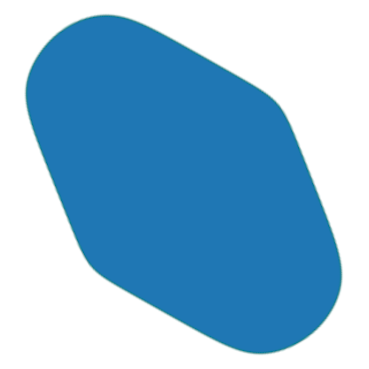}
        \caption*{$\lambda_4 = 65.5689$\\$E_{\text{rel}} = 4.8e-3$}
    \end{subfigure}
    \begin{subfigure}{0.26\textwidth}
        \centering
        \includegraphics[width=\linewidth]{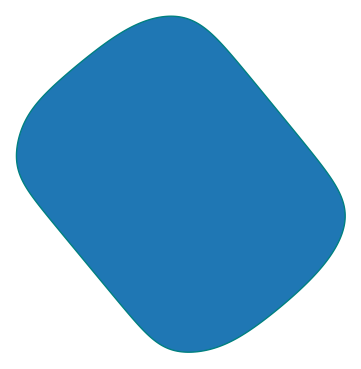}
        \caption*{$\lambda_5 = 80.0961$\\$E_{\text{rel}} = 4.8e-3$}
    \end{subfigure}
    \begin{subfigure}{0.26\textwidth}
        \centering
        \includegraphics[width=\linewidth]{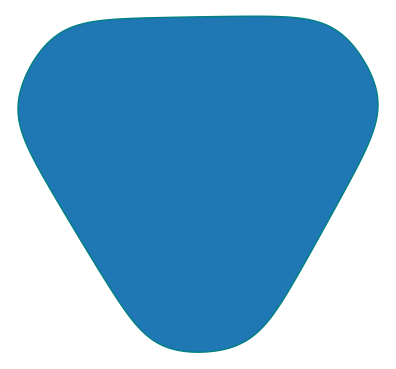}
        \caption*{$\lambda_6 = 88.9726$\\$E_{\text{rel}} = 4.9e-3$}
    \end{subfigure}
    
    \caption{\label{fig:convex_dirichlet} Optimal shapes for the first six Dirichlet eigenvalues}
\end{figure}


\end{document}